\documentclass[11pt]{article}

\usepackage{inputenc}%[latin1]
\usepackage{bbm}
\usepackage{amsmath}
\usepackage{amsfonts}
\usepackage{amssymb}
\usepackage{amsthm}
\usepackage{comment}
\usepackage{graphicx}
\usepackage[T1]{fontenc}
\usepackage[english]{babel}
\usepackage{mathrsfs}
\usepackage{amsthm}
\usepackage[all]{xy}
\usepackage{geometry}
\usepackage{enumerate}
\usepackage{pgfplots}
\usepackage{subfigure}
\pgfplotsset{compat=1.15}
\allowdisplaybreaks
\usepackage{xcolor}
\usepackage{appendix}
\usepackage{authblk}
\usepackage[colorlinks=false]{hyperref} 
\usepackage{cleveref}
\usepackage{dsfont}
\usepackage{mathtools}

\newtheorem{theorem}{Theorem}
\newtheorem{corollary}[theorem]{Corollary}
\newtheorem{lemma}[theorem]{Lemma}
\newtheorem{definition}[theorem]{Definition}
\newtheorem{proposition}[theorem]{Proposition}

\newtheorem{definition-proposition}[theorem]{Definition-Proposition}

\newtheorem{remark}[theorem]{Remark}

\DeclareMathOperator{\supp}{supp}
\DeclareMathOperator{\meas}{meas}

\newcommand{\N} {\mathbb{N}}
\newcommand{\Z} {\mathbb{Z}}

\newcommand{\R} {\mathbb{R}}

\DeclarePairedDelimiter\abs{\lvert}{\rvert}%
\DeclarePairedDelimiter\norm{\lVert}{\rVert}%
\newcommand{\tnorm}[1]{{\left\vert\kern-0.25ex\left\vert\kern-0.25ex\left\vert #1 
    \right\vert\kern-0.25ex\right\vert\kern-0.25ex\right\vert}}

\makeatletter
\let\oldabs\abs
\def\abs{\@ifstar{\oldabs}{\oldabs*}}
\let\oldnorm\norm
\def\norm{\@ifstar{\oldnorm}{\oldnorm*}}
\makeatother

\newcommand*\diff{\mathop{}\!\mathrm{d}}

\title{Existence and Uniqueness of Domain Walls for Notched Ferromagnetic Nanowires}
\author[1]{Raphaël Côte}
\author[1]{Clémentine Courtès}
\author[2]{Guillaume Ferrière}
\author[3]{Ludovic Godard-Cadillac}
\author[4,5]{Yannick Privat}

\affil[1]{\footnotesize IRMA, Universit\'e de Strasbourg, CNRS UMR 7501, Inria, 7 rue Ren\'e Descartes, 67084 Strasbourg (France)\smallskip}
\affil[2]{\footnotesize Univ. Lille, Inria, CNRS, UMR 8524 - Laboratoire Paul Painlevé, F-59000 Lille (France)\smallskip}
\affil[3]{\footnotesize Institut Mathématiques de Bordeaux, Bordeaux-INP, UMR 525, 351 cours de la Libération - F 33 405 Talence (France)\smallskip}
\affil[4]{\footnotesize Institut Élie Cartan de Lorraine, École des Mines de Nancy, Boulevard des Aiguillettes, 54506 Vandoeuvre-lès-Nancy (France)}
\affil[5]{\footnotesize Institut Universitaire de France (IUF)}

\begin{document}
\maketitle

\begin{abstract}
In this article, we investigate a simple model of notched ferromagnetic nanowires using tools from calculus of variations and critical point theory. Specifically, we focus on the case of a single unimodal notch and establish the existence and uniqueness of the critical point of the energy. This is achieved through a lifting argument, which reduces the problem to a generalized Sturm-Liouville equation.

Uniqueness is demonstrated via a Mountain-Pass argument, where the assumption of two distinct critical points leads to a contradiction. Additionally, we show that the solution corresponds to a system of magnetic spins characterized by a single domain wall localized in the vicinity of the notch. We further analyze the asymptotic decay of the solution at infinity and explore the symmetric case using rearrangement techniques.
\end{abstract}


\medskip


%\tableofcontents

%%%%%%%%%%%%%%%%%%%%%%%%%%%%%%%%%%%%%%%%%%%%%%%%%%%%%%%%%%%%%%
%%%%%%%%%%%%%%%%%%%%%%%%%%%%%%%%%%%%%%%%%%%%%%%%%%%%%%%%%%%%%%
%%%%%%%%%%%%%%%%%%%%%%%%%%%%%%%%%%%%%%%%%%%%%%%%%%%%%%%%%%%%%%
\section{Introduction}

\subsection{Presentation of ferromagnetic nanowires}
Ferromagnetic nanowires are nanoscale devices with significant potential for applications in micromagnetic engineering. These include high-density data storage, magnetic logic gates, microelectronics, radar stealth coatings, transformers and so on (\cite{Parkin_Hayashi_Thomas_2008}).
A ferromagnetic nanowire is a nanoscale crystal of ferromagnetic atoms, characterized by a cylindrical shape with a small cross-section and a much greater length in one direction. The ferromagnetic interactions among the crystal's atoms cause the magnetic spins to align with one another, while the shape anisotropy of the crystal promotes alignment along the wire's principal axis. When the magnetization points in different directions at the wire's ends, the magnetic spins must reverse orientation in the middle of the wire to satisfy the topological constraint.
This narrow region where the magnetization changes rapidly is known as a Néel wall. Its stability and dynamics have been the subject of extensive research; see, for example,~\cite{Carbou_Jizzini_2015, Carbou_Labbe_2006, Cote_Ferriere__2DW, Cote_Ignat__stab_DW_LLG_DM, Ferriere_preprint, Jizzini_2011, Keisuke_2011, Labbe_Privat_Trelat_2012, Slatiskov_Sonnenberg_2012, Thiaville_Nakatani_2006} and references therein.
Control theory for Néel walls has also been extensively studied and developed~\cite{Carbou_Labbe_2012,Carbou_Labbe_Trelat_2007,MR3348399}. 
%For a larger introduction to micromagnetism, we refer to
For a more comprehensive introduction to micromagnetism, we refer readers to ~\cite{hubert2008magnetic}.

In nanowires designed for industrial applications, it is common to introduce defects to stabilize the Néel wall near a specific defect. One such defect involves adding notches to the nanowire, which are small regions with reduced radius, thereby decreasing the exchange interaction locally. In this article, we study a simplified model introduced in~\cite{book_carbou,carbou2018stabilization},  which is based on a 1D modified Sturm-Liouville equation. The existence of a solution was demonstrated using a shooting method, while a uniqueness result was later established in~\cite{Ignat_2022}.

The primary aim of this work is to reinterpret the results of \cite{book_carbou,carbou2018stabilization} using the framework of the calculus of variations and a minimization problem, making the approach more aligned with the underlying physics. We provide a new proof of existence and establish a more general uniqueness result based on a detailed analysis of the critical points of the associated energy.

\subsection{A model for a notched ferromagnetic nanowire}\label{sec:remind}

We are interested in a model of a straight notched ferromagnetic nanowire. The direction of the nanowire is assumed to be $\mathbf{e_1}$ where 
\begin{equation*}
    \mathbf{e_1}= \begin{pmatrix} 
    1 \\ 0 \\ 0
    \end{pmatrix}, \quad
    \mathbf{e_2}= \begin{pmatrix} 
    0 \\ 1 \\ 0
    \end{pmatrix}, \quad
    \mathbf{e_3} = \begin{pmatrix} 
    0 \\ 0 \\ 1
    \end{pmatrix}
\end{equation*}
is the canonical basis of $\mathbb{R}^3$. The magnetization $\boldsymbol{m} = (m_1, m_2, m_3): \mathbb{R} \to \mathbb{S}^2$ of this nanowire takes its values into the unit sphere $\mathbb{S}^2\subset \mathbb{R}^3$.
We will consider models introduced in \cite{book_carbou,carbou2018stabilization}, in which the magnetization behavior is obtained due to a $\Gamma$-convergence reasoning: a cylindrical material $\mathcal{D}_\eta$ is considered, given in $\mathscr{B}$ by
\begin{equation*}
\mathcal{D}_\eta=\{(x,y,z)\in [-L,L]\times \R^2, \ y^2+z^2\leq \eta ^2\rho(x)^2\},
\end{equation*}
whose circular section, parameterized by a function $\rho$, has radius $\eta \rho (x)$ with $\eta >0$. 
	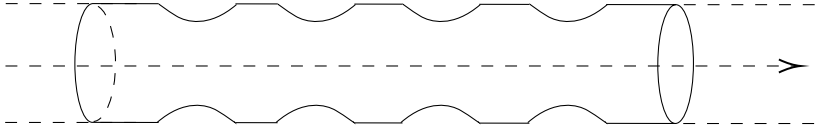
\begin{figure}[h!]
	\centering
\begin{tikzpicture}[x=0.75pt,y=0.75pt,yscale=-1,xscale=1]
%uncomment if require: \path (0,310); %set diagram left start at 0, and has height of 310

%Shape: Ellipse [id:dp06285673876443965] 
\draw  [dash pattern={on 4.5pt off 4.5pt}] (153,93.61) .. controls (153.02,77.38) and (158.08,64.24) .. (164.29,64.25) .. controls (170.5,64.26) and (175.52,77.42) .. (175.49,93.64) .. controls (175.47,109.86) and (170.42,123.01) .. (164.21,123) .. controls (157.99,122.99) and (152.98,109.83) .. (153,93.61) -- cycle ;
%Shape: Rectangle [id:dp42386107449879995] 
\draw  [color={rgb, 255:red, 255; green, 255; blue, 255 }  ,draw opacity=1 ][fill={rgb, 255:red, 255; green, 255; blue, 255 }  ,fill opacity=1 ] (149.2,58.76) -- (164.2,58.76) -- (164.2,123) -- (149.2,123) -- cycle ;

%Shape: Can [id:dp7728844235395792] 
\draw   (457.83,124.02) -- (163.85,123.45) .. controls (158.88,123.44) and (154.92,110.02) .. (155,93.46) .. controls (155.08,76.9) and (159.18,63.49) .. (164.14,63.5) -- (458.13,64.06) .. controls (463.09,64.07) and (467.06,77.5) .. (466.97,94.06) .. controls (466.89,110.61) and (462.8,124.03) .. (457.83,124.02) .. controls (452.87,124.01) and (448.91,110.58) .. (448.99,94.02) .. controls (449.07,77.47) and (453.16,64.05) .. (458.13,64.06) ;
%Shape: Rectangle [id:dp0505509185505697] 
\draw  [color={rgb, 255:red, 255; green, 255; blue, 255 }  ,draw opacity=1 ][fill={rgb, 255:red, 255; green, 255; blue, 255 }  ,fill opacity=1 ] (320.5,44.07) -- (358.73,44.07) -- (358.73,144.5) -- (320.5,144.5) -- cycle ;
%Curve Lines [id:da9506428117678422] 
\draw    (320.5,64) .. controls (319.5,63) and (324.5,67) .. (329.29,69.85) .. controls (334.08,72.7) and (345.5,76) .. (359.5,64) ;
%Curve Lines [id:da6860729817188593] 
\draw    (358.41,123.19) .. controls (360.5,125) and (354.38,120.23) .. (349.56,117.43) .. controls (344.75,114.62) and (333.3,111.44) .. (319.41,123.57) ;

%Straight Lines [id:da5296249523987883] 
\draw  [dash pattern={on 4.5pt off 4.5pt}]  (186.14,63.52) -- (116.14,63.52) ;
%Straight Lines [id:da9243202690156924] 
\draw  [dash pattern={on 4.5pt off 4.5pt}]  (185.85,123.5) -- (115.85,123.5) ;
%Straight Lines [id:da5019934709264636] 
\draw  [dash pattern={on 4.5pt off 4.5pt}]  (528.12,64.06) -- (458.12,64.06) ;
%Straight Lines [id:da19643549486628442] 
\draw  [dash pattern={on 4.5pt off 4.5pt}]  (527.83,124.04) -- (457.83,124.04) ;
%Shape: Rectangle [id:dp04625436492840407] 
\draw  [color={rgb, 255:red, 255; green, 255; blue, 255 }  ,draw opacity=1 ][fill={rgb, 255:red, 255; green, 255; blue, 255 }  ,fill opacity=1 ] (384.5,44.07) -- (422.73,44.07) -- (422.73,144.5) -- (384.5,144.5) -- cycle ;
%Curve Lines [id:da693824841660731] 
\draw    (384.5,64) .. controls (383.5,63) and (388.5,67) .. (393.29,69.85) .. controls (398.08,72.7) and (409.5,76) .. (423.5,64) ;
%Curve Lines [id:da41424986398837416] 
\draw    (422.41,123.19) .. controls (424.5,125) and (418.38,120.23) .. (413.56,117.43) .. controls (408.75,114.62) and (397.3,111.44) .. (383.41,123.57) ;

%Shape: Rectangle [id:dp8904767304563893] 
\draw  [color={rgb, 255:red, 255; green, 255; blue, 255 }  ,draw opacity=1 ][fill={rgb, 255:red, 255; green, 255; blue, 255 }  ,fill opacity=1 ] (257.5,44.07) -- (295.73,44.07) -- (295.73,144.5) -- (257.5,144.5) -- cycle ;
%Curve Lines [id:da0385496869970573] 
\draw    (257.5,64) .. controls (256.5,63) and (261.5,67) .. (266.29,69.85) .. controls (271.08,72.7) and (282.5,76) .. (296.5,64) ;
%Curve Lines [id:da8939317119286189] 
\draw    (295.41,123.19) .. controls (297.5,125) and (291.38,120.23) .. (286.56,117.43) .. controls (281.75,114.62) and (270.3,111.44) .. (256.41,123.57) ;

%Shape: Rectangle [id:dp34974216193705165] 
\draw  [color={rgb, 255:red, 255; green, 255; blue, 255 }  ,draw opacity=1 ][fill={rgb, 255:red, 255; green, 255; blue, 255 }  ,fill opacity=1 ] (197.5,44.07) -- (235.73,44.07) -- (235.73,144.5) -- (197.5,144.5) -- cycle ;
%Curve Lines [id:da050952196021834206] 
\draw    (197.5,64) .. controls (196.5,63) and (201.5,67) .. (206.29,69.85) .. controls (211.08,72.7) and (222.5,76) .. (236.5,64) ;
%Curve Lines [id:da8479050124731089] 
\draw    (235.41,123.19) .. controls (237.5,125) and (231.38,120.23) .. (226.56,117.43) .. controls (221.75,114.62) and (210.3,111.44) .. (196.41,123.57) ;

%Straight Lines [id:da9356323426008402] 
\draw  [dash pattern={on 4.5pt off 4.5pt}]  (120,95) -- (519,95) ;
\draw [shift={(521,95)}, rotate = 180] [color={rgb, 255:red, 0; green, 0; blue, 0 }  ][line width=0.75]    (10.93,-3.29) .. controls (6.95,-1.4) and (3.31,-0.3) .. (0,0) .. controls (3.31,0.3) and (6.95,1.4) .. (10.93,3.29)   ;
\end{tikzpicture}
	\caption{An example of domain $\mathcal{D}_\eta$.}
	\end{figure}

A 1D model is then derived by making $\eta$ tend towards 0. The 1D model involves the cross section area $s$ defined by $s(x)=\pi \rho(x)^2$ and the nanowire domain becomes $\Omega = \mathbb{R} \mathbf{e_1} \subset \mathbb{R}^3$,

Let us denote by $\alpha>0$ the gyromagnetic ratio and by $\ell>0$ the exchange length. 
In what follows, $\Omega$ stands for the nanowire domain and $s\in W^{1,\infty}(\Omega)$ for the residual cross-section area of the nanowire, or in other words the shape of notches on the nanowire.
The asymptotic Landau-Lifshitz model for magnetization in notched nanowires, provided in \cite{book_carbou,carbou2018stabilization}, reads

\begin{equation}\label{eq:LLG}\tag{LLG}
\left\{\begin{array}{ll}
\partial_t\boldsymbol{m}=-\boldsymbol{m}\times {H}(\boldsymbol{m})-\alpha \boldsymbol{m}\times (\boldsymbol{m}\times {H}(\boldsymbol{m})) \\
{H}(\boldsymbol{m})=\frac{\ell^2}{s(x)}\partial_x \left(s(x)\partial_x\boldsymbol{m}\right)-(m_2\mathbf{e_2}+m_3\mathbf{e_3}), & \\
\end{array}
\right.
\end{equation}
To avoid the nanowire collapse, it is assumed that
\begin{equation*}
\exists s_0>0 \textrm{ such that }s(\cdot) \ge s_0\textrm{ in }\Omega.
\end{equation*}
We will also impose a uniform upper-bound on $s$. Up to using a renormalization argument by dividing $s(x)$ by some $s_1 > 0$ (it is important to note that $\boldsymbol{m}$ still satisfies \eqref{eq:LLG} when we change $s(x)$ by $\frac{s (x)}{s_1}$), that will be fixed equal to 1:
\begin{equation*}
s(\cdot)\leq 1\textrm{ in }\Omega.
\end{equation*}
so that in particular, we must also assume that $s_0<1$.

It will be convenient to consider cross sections $s$ that are compactly supported :
\begin{equation*}
\mathscr{S}_a(\Omega)= \{s\in W^{1,\infty}(\Omega,[s_0,1])\mid s(x)=1\text{ in }\Omega\setminus [-a,a], \, s \geq s_0 \text{ in }\Omega\}.
\end{equation*}

The \eqref{eq:LLG} flow is equivariant under rotations $R_\varphi \coloneqq \begin{pmatrix} 
1 & 0 & 0 \\
0 & \cos \varphi & -\sin \varphi \\
0 & \sin \varphi & \cos \varphi
\end{pmatrix}$ around the axis $\mathbf{e_1}$ with angle $\varphi\in \mathbb R$. This transformation preserves $\mathbb S^2$ valued functions, and so acts on magnetization; it also extends naturally to functions of space and time, for which it preserves solutions to \eqref{eq:LLG}.

\subsection{Energy, functional spaces and Cauchy theory}

From this formulation of the Landau-Lifschitz-Gilbert equation, it is possible to write an energetic problem. Using a change of variable, we can reduce the problem to $\ell = 1$ (which simplifies the notations). We now observe that satisfying \eqref{eq:LLG} is equivalent to being a critical point of the following energy: 

\begin{equation}\label{expr:SS_infinite:Energy_m}
    E_s(\boldsymbol{m}):=\frac{1}{2}\int_\R \abs{\partial_x \boldsymbol{m} (x)}^2 s(x)\, \diff x
    +\frac{1}{2} \int_\R (m_2^2 (x) + m_3^2 (x)) \,s(x)\, \diff x.
\end{equation}

To ensure that this energy has meaning, we introduce the energy space $\mathcal{H}^1$
\begin{equation*}
    \mathcal{H}^1 \coloneqq \{ f = (f_1, f_2, f_3) \in H^1_\textnormal{loc} (\mathbb R) \mid f', f_2, f_3 \in L^2 (\mathbb R) \text{ and } \abs{f} = 1 \text{ a.e.} \},
\end{equation*}
which is simply the same energy space on $\R$ as in \cite{Cote_Ignat__stab_DW_LLG_DM}, except that we add the weight $s (x) \diff x$ in the norm, so that
\begin{equation*}
    2 E_s(\boldsymbol{m}) = \int_\R \abs{\partial_x \boldsymbol{m}}^2 s(x)\, \diff x+\int_\R (m_2^2 + m_3^2) s(x)\, \diff x \eqqcolon \|\boldsymbol{m}\|_{\mathcal{H}^1}^2.
\end{equation*}
Similarly are defined the spaces $\mathcal{H}^k$ and $\mathcal{H}^\infty$.
To the best of our knowledge, the Cauchy theory for one-dimensional Landau-Lifshitz-Gilbert (LLG) equations with notches has not yet been developed.

\subsection{Domain walls for notched nanowire: a variational formulation}
Throughout this article we call a \textit{domain walls system} or a \textit{Néel walls system} any stationary solution $\boldsymbol{m}$ to the uni-dimensional LLG equation (with or without notch) that is connecting the two constant states $-\mathbf{e_1}$ and $+\mathbf{e_1}$. Without restriction of generality, we will focus of the case :
$$\boldsymbol{m}(-\infty)=-e_1\qquad\text{and}\qquad\boldsymbol{m}(+\infty)=+e_1.$$
The steady-states $\boldsymbol{m} \in \mathcal{H}^1$ for this system solve the equation
\begin{equation}\label{eqLLSS:infinite}
\left\{\begin{array}{ll}
\boldsymbol{m}\times {H}(\boldsymbol{m})=0 & \text{on } \R\\
{H}(\boldsymbol{m})=\frac{1}{s(x)}\partial_x \left(s(x)\partial_x\boldsymbol{m}\right)-(m_2\mathbf{e_2}+m_3\mathbf{e_3}). &
\end{array}
\right.
\end{equation}

It has been proved in \cite{book_carbou,carbou2018stabilization, Ignat_2022}  that every steady solution in \eqref{eqLLSS:infinite} lays in a single plane and then, up to a constant rotation of angle $\varphi$, reads
\begin{equation}\label{expr:SS_infinite}
\boldsymbol{m}(x)=R_\varphi \begin{pmatrix}
\sin\theta(x)\\ \cos \theta(x)\\ 0
\end{pmatrix}, \quad \text{with}\quad R_\varphi=\begin{pmatrix}
1 & 0 & 0\\
0 & \cos \varphi & -\sin \varphi \\
0 & \sin \varphi & \cos \varphi
\end{pmatrix},
\end{equation}
where $\theta$ solves the non-linear Sturm-Liouville equation 

\begin{equation} \label{expr:SSTheta_infinite}
    \theta''(x)+\frac{s'(x)}{s(x)}\theta'(x)+\cos \theta(x)\sin \theta(x)=0, \qquad x\in \R,
\end{equation}
or, in another way,
\begin{equation}\label{expr:SS_infinite:div a grad formulation}
    \left(s(x)\theta'(x)\right)'+s(x)\,\cos \theta(x)\sin \theta(x)=0\quad \text{in }\R .
\end{equation}
This last expression leads to the variational formulation of the equation:
\begin{equation*}
    \forall\;\varphi\in{\mathcal D}(\R),\qquad\int_\R\theta'(x)\,\varphi'(x)\,s(x)\diff x=\int_\R\cos \theta(x)\sin \theta(x)\,\varphi(x)\,s(x)\diff x.
\end{equation*}

It is natural to expect the presence of a "single domain wall," which refers to a unique reversal of the spin from $-\mathbf{e_1}$ to $-\mathbf{e_1}$ when $\boldsymbol{m}$ is the global minimizer of $E_s$. This configuration, simply called a \emph{domain wall}, holds for notchless nanowires but may not apply if the geometry of the notch does not support such a configuration.

Before proceeding to state our theorem, we must first provide a rigorous definition of the concept of a "single domain wall," which we outline here:
\begin{definition}\label{def:domWall}
    A solution to~\eqref{eq:LLG} is said to be a \textit{domain wall} if it is stationary, local minimizer of the energy $E_s$ in \eqref{expr:SS_infinite:Energy_m}, and the associated angle $\theta$ define at~\eqref{expr:SS_infinite} is monotonous and satisfy $\theta(\pm\infty)=\pm\pi/2$.
\end{definition}

\subsection{A minimization problem for infinite unimodal nanowire}
Our main result focuses on the case of an infinite nanowire with a single notch. To demonstrate that the minimizer corresponds to a single domain wall, we restrict our analysis to a specific class of notches, namely unimodal notches:
\begin{equation}\label{def:classU}
\mathscr{U}_a(\mathbb{R})= \{s\in \mathscr{S}_a(\mathbb{R})\mid s\text{ is non-increasing on } [-a, 0] \text{ and non-decreasing on }[0,a]\}.
\end{equation}
Such a notch decreases the exchange interaction between ferromagnetic atoms and we prove in this article that in this case we have a ``single domain wall'' that is localized in a close neighborhood of the notch. If we have good symmetry properties, we can prove more precise results. For that purpose we introduce the symmetric notches :
\begin{equation*}%\label{def:classA}
\mathscr{A}_a(\Omega)= \{s\in \mathscr{S}_a(\Omega)\mid s\text{ is even and non-decreasing on }[0,a]\}.
\end{equation*}
In particular, $\mathscr{A}_a (\Omega) \subset \mathscr{U}_a (\Omega)$.
\begin{figure}[h!]
\begin{center}
\begin{tikzpicture}[x=0.75pt,y=0.75pt,yscale=-1,xscale=1]
%uncomment if require: \path (0,310); %set diagram left start at 0, and has height of 310

%Shape: Ellipse [id:dp9393495531940262] 
\draw  [dash pattern={on 4.5pt off 4.5pt}] (165,54.15) .. controls (165.02,37.67) and (170.08,24.33) .. (176.29,24.33) .. controls (182.5,24.34) and (187.52,37.71) .. (187.5,54.18) .. controls (187.47,70.66) and (182.42,84.01) .. (176.21,84) .. controls (169.99,83.99) and (164.98,70.63) .. (165,54.15) -- cycle ;
%Shape: Rectangle [id:dp4029289620411871] 
\draw  [color={rgb, 255:red, 255; green, 255; blue, 255 }  ,draw opacity=1 ][fill={rgb, 255:red, 255; green, 255; blue, 255 }  ,fill opacity=1 ] (161.2,18.76) -- (176.2,18.76) -- (176.2,84) -- (161.2,84) -- cycle ;

%Shape: Can [id:dp5966092833958019] 
\draw   (446.83,85.04) -- (174.85,84.5) .. controls (169.88,84.49) and (165.92,71.05) .. (166,54.49) .. controls (166.08,37.93) and (170.17,24.51) .. (175.14,24.52) -- (447.12,25.06) .. controls (452.09,25.07) and (456.05,38.51) .. (455.97,55.07) .. controls (455.89,71.63) and (451.8,85.05) .. (446.83,85.04) .. controls (441.86,85.03) and (437.9,71.59) .. (437.98,55.03) .. controls (438.06,38.47) and (442.16,25.05) .. (447.12,25.06) ;
%Shape: Rectangle [id:dp7638128803452042] 
\draw  [color={rgb, 255:red, 255; green, 255; blue, 255 }  ,draw opacity=1 ][fill={rgb, 255:red, 255; green, 255; blue, 255 }  ,fill opacity=1 ] (291.5,5.07) -- (329.73,5.07) -- (329.73,105.5) -- (291.5,105.5) -- cycle ;
%Curve Lines [id:da767882949793648] 
\draw    (291.5,25) .. controls (290.5,24) and (295.5,28) .. (300.29,30.85) .. controls (305.08,33.7) and (316.5,37) .. (330.5,25) ;
%Straight Lines [id:da16237726958447785] 
\draw  [dash pattern={on 4.5pt off 4.5pt}]  (109,56) -- (508,56) ;
\draw [shift={(510,56)}, rotate = 180] [color={rgb, 255:red, 0; green, 0; blue, 0 }  ][line width=0.75]    (10.93,-3.29) .. controls (6.95,-1.4) and (3.31,-0.3) .. (0,0) .. controls (3.31,0.3) and (6.95,1.4) .. (10.93,3.29)   ;
%Curve Lines [id:da34155132408702704] 
\draw    (329.41,84.19) .. controls (331.5,86) and (325.38,81.23) .. (320.56,78.43) .. controls (315.75,75.62) and (304.3,72.44) .. (290.41,84.57) ;
%Straight Lines [id:da4447665737565393] 
\draw  [dash pattern={on 4.5pt off 4.5pt}]  (175.14,24.52) -- (105.14,24.52) ;
%Straight Lines [id:da8073849331635697] 
\draw  [dash pattern={on 4.5pt off 4.5pt}]  (174.85,84.5) -- (104.85,84.5) ;
%Straight Lines [id:da8140472815786994] 
\draw  [dash pattern={on 4.5pt off 4.5pt}]  (517.12,25.06) -- (447.12,25.06) ;
%Straight Lines [id:da17709192212512714] 
\draw  [dash pattern={on 4.5pt off 4.5pt}]  (516.83,85.04) -- (446.83,85.04) ;
\end{tikzpicture}
	\caption{An infinite nanowire with one notch.}
\end{center}
\end{figure}
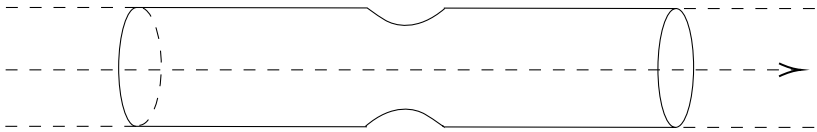

The energy $E_s$ represents the magnetic energy of the nanowire, and it is particularly relevant for the study of domain walls. It is always non-negative, making it suitable for minimization problems. However, it is crucial to correctly define the functional space on which the minimization is performed to ensure the problem is non-trivial. Indeed, if we attempt to minimize directly within the space $\mathcal{H}^1$, we find that $\min_{\boldsymbol{m} \in \mathcal{H}^1} E_s (\boldsymbol{m}) = 0$, with the constant solutions $\mathbf{e_1}$ and $-\mathbf{e_1}$ as minimizers.

To make the problem physically meaningful, boundary conditions at $\pm \infty$ must be imposed. This additional constraint makes the problem non-trivial. For example, in the case of a notchless nanowire ($s \equiv 1$), it is well-known that the domain wall solution $w_*$ is the unique minimizer (up to rotations and translations) of $E_1(\boldsymbol{m})$ in $\mathcal{H}^1$, under the condition that $\lim_{\pm \infty} \boldsymbol{m} = \pm \mathbf{e_1}$.

We will now apply this framework to study the following minimization problem in the case of a nanowire with a notch:
\begin{equation} \label{eq:min_prob}
    \mathfrak{m} \coloneqq \inf_{\boldsymbol{m} \in \mathcal{H}^1_{\neq}} E_s (\boldsymbol{m}),
\end{equation}
where
\begin{equation*}
    \mathcal{H}^1_{\neq} \coloneqq \{ m \in \mathcal{H}^1 \; : \; \lim_{- \infty} m = - \mathbf{e_1} \text{ and } \lim_{\infty} m = \mathbf{e_1} \}.
\end{equation*}

Let us now state our main result:

\begin{theorem}
    For any $s \in \mathscr{U}_a(\mathbb{R})$ such that $s(\cdot) \neq 1$, there exists a \emph{unique} domain wall\footnote{see Definition~\ref{def:domWall}} $\boldsymbol{w}_s$ connecting $- \mathbf{e_1}$ to $+ \mathbf{e_1}$ in the following sense:
    \begin{itemize}
        \item There exists a unique $\theta_s \in H^1_\textnormal{loc} (\mathbb{R})$ satisfying
        \begin{itemize}
            \item $\lim_{x \rightarrow \pm \infty} \theta_s (x) = \pm \frac{\pi}{2}$,
            \item The function $\displaystyle 
                    \boldsymbol{w}_s \coloneqq 
                    \begin{pmatrix}
                        \sin \theta_s \\
                        \cos \theta_s \\
                        0
                    \end{pmatrix}$
            is a steady state of \eqref{eq:LLG} which belongs to $\mathcal{H}^1_{\neq}$. 
        \end{itemize}
       \item The function $\theta_s$ is increasing and satisfy the following decay estimate :
       $$    \forall\;x\in\R,\qquad\bigg|\big|\theta_s(x)\big|-\frac{\pi}{2}\bigg|\;\leq\; \pi \,\exp\bigg(-\int_0^{|x|}\frac{dy}{s(y)}\bigg).$$
        \item If $s\in\mathscr{A}$ then $\theta_s$ is odd.
        
        \item Up to rotations $R_\varphi$ around $\mathbf{e_1}$, $\boldsymbol{w}_s$ is the unique minimizer of the minimization problem \eqref{eq:min_prob}.
    \end{itemize}
\end{theorem}
If $s \in \mathscr{A}_a(\mathbb{R})$, we recover the existence and symmetry results for the domain wall with a new proof compared to \cite{book_carbou,carbou2018stabilization}. Concerning uniqueness, the main improvement of our result compared to~\cite{Ignat_2022} lays in the fact that we do not need to work up to translations and assume $\theta_s(0)=0$.

In fact, this hypothesis concerns both earlier results \cite{carbou2018stabilization,Ignat_2022}. We however feel that this hypothesis is not very natural and rather restrictive, both for the existence and the uniqueness questions. 

Regarding existence, assuming $\theta_0(0) =0$ corresponds to solving a shooting problem; when $s$ is non symmetric, there is no reason that a domain wall satisfy this constraint. Similarly for uniqueness, there is no reason that two domains walls reach $0$ at the same point.

Our main hypothesis is that $s$ is unimodal as mentioned in \eqref{def:classU}. This class of notch is natural both for applications and for uniqueness. Indeed, one can construct $s$ with 2 notches, such that two domain walls exist (or even simply, in the case of a periodic $s$, infinitely many domains can be obtained by translation of a period). One of the main point of this paper is that (under the unimodal assumption), the presence of a non trivial notch rules out minimizing sequences where the transition from $\mathbf{e_1}$ to $-\mathbf{e_1}$ occurs at spatial infinity, and allows to recover compactness.

\bigskip

The proof of the main theorem relies on the calculus of variations, partially inspired by \cite{Ignat_2022}, but the key part of our argument differs. Indeed, the most challenging part concerns the uniqueness property of the domain wall. This is established under the assumption of unimodality of the notch.
It involves assuming the existence of two distinct critical points and constructing a third one using the mountain-pass theorem. This construction leads to a contradiction with a previously established stability result in \cite{book_carbou,carbou2018stabilization}, thereby concluding the proof. 

The rest of the article is organized as follows: Section~\ref{sec:exists} focuses on proving the existence of a minimizer to Problem~\eqref{eq:min_prob}, with the main result detailed in Proposition~\ref{prop:strict monotony}. In Section~\ref{sec:uniqueness}, we demonstrate that the functional $E_s$ has a unique critical point on $\mathcal{H}^1_{\neq}$. The main result of this section is presented in Theorem~\ref{theo:criticalPt}.

\section{Existence of the domain wall}\label{sec:exists}

\subsection{A useful change of variables}

In the following sections, we will sometimes use the following standard change of variable for Sturm-Liouville equation (see e.g. \cite{cox1996extremal})
\begin{equation} \label{eq:def_y_change_var}
    y:\;x\;\longmapsto\;\int_0^x\frac{du}{s(u)}.
\end{equation}
This function is a ${\mathcal C}^{0,1}$-diffeomorphism on $\R$ since we have $s(x)\geq s_0 >0$.
Thus, for any function $f$, we can define a new function $g$ through this change of variable: $g(y) = f(x)$. Saying it in another way, $g \coloneqq f \circ y^{-1}$.
In particular, we also define $\sigma$ such that $\sigma(y) = s(x)$, which is constant equal to $1$ outside $(\mathfrak{a}_-, \mathfrak{a}_+)$ where $\mathfrak{a}_\pm \coloneqq y (\pm a)$.
With such a change of variable, the steady-state of \eqref{eq:LLG} gain a bit of regularity.

\begin{lemma} \label{lem:chg_var_m}
    For any steady-state $\boldsymbol{m} \in \mathcal{H}^1$ (solution of \eqref{eqLLSS:infinite}), the function $\tilde{\boldsymbol{m}} = \boldsymbol{m} \circ y^{-1} \in \mathcal{H}^1$ satisfies $\partial_{yy} \tilde{\boldsymbol{m}} \in L^2 + L^1$ and $\tilde{\boldsymbol{m}} \in \mathcal{C}^{1}$.
\end{lemma}

\begin{proof}
    It is easy to compute that
    \begin{align*}
        \partial_y \tilde{\boldsymbol{m}} (y) &= s(x) \partial_x \boldsymbol{m} (x), \\
        \partial_{yy} \tilde{\boldsymbol{m}} (y) &= s(x) \partial_x \Bigl( s(x) \partial_x \boldsymbol{m} \Bigr) (x), \\
        H(m) (x) &= \frac{1}{\sigma (y)^2} \partial_{yy} \tilde{\boldsymbol{m}} (y) + \tilde{\boldsymbol{m}}_2 (y) e_2 + \tilde{\boldsymbol{m}}_3 (y) e_3 = \frac{1}{\sigma (y)^2} \tilde{H} (\tilde{\boldsymbol{m}}) (y),
    \end{align*}
    where
    \begin{equation*}
        \tilde{H} (\tilde{\boldsymbol{m}}) (y) \coloneqq \partial_{yy} \tilde{\boldsymbol{m}} (y) + \sigma^2 (y) \Bigl( \tilde{m}_2 (y) e_2 + \tilde{m}_3 (y) e_3 \Bigr).
     \end{equation*}
    Thus, the equation \eqref{eqLLSS:infinite} that $\boldsymbol{m}$ satisfies implies for $\tilde{\boldsymbol{m}}$:
    \begin{equation*}
        \tilde{\boldsymbol{m}} \times \tilde{H} (\tilde{\boldsymbol{m}}) = 0.
    \end{equation*}
    Since $\abs{\tilde{\boldsymbol{m}}} = 1$, which implies that $\partial_{yy} \tilde{\boldsymbol{m}} \cdot \tilde{\boldsymbol{m}} = - \abs{\partial_y \tilde{\boldsymbol{m}}}^2$, this equation can be redrafted as
    \begin{equation} \label{eq:ODE_tilde_m}
        \tilde{H} (\tilde{\boldsymbol{m}}) = \tilde{\Lambda} \, \tilde{\boldsymbol{m}},
    \end{equation}
    where
    \begin{equation*}
        \tilde{\Lambda} = \tilde{H} (\tilde{\boldsymbol{m}}) \cdot \tilde{\boldsymbol{m}} = - \abs{\partial_y \tilde{\boldsymbol{m}}}^2 + \sigma^2 (y) \Bigl( \tilde{m}_2^2 + \tilde{m}_3^2 \Bigr) \in L^1.
    \end{equation*}
    This leads to $\tilde{\Lambda} \, \tilde{\boldsymbol{m}} \in L^1$ as $\tilde{\boldsymbol{m}} \in L^\infty$.
    Since $\sigma \in L^\infty$ and $\tilde{m}_2, \tilde{m}_3 \in L^2$, we get that $\sigma^2 (y) \Bigl( \tilde{m}_2 (y) e_2 + \tilde{m}_3 (y) e_3 \Bigr) \in L^2$, thus $\partial_{yy} \tilde{\boldsymbol{m}} \in L^1 + L^2$. The conclusion follows.
\end{proof}

\subsection{Steady-states are planar}
In this section we rewrite the arguments presented in~\cite{book_carbou,carbou2018stabilization} to obtain a problem written in term of Sturm-Liouville equation on the angle $\theta$.
\begin{lemma} \label{lem:lift_planar}
    For any $\boldsymbol{m} \in \mathcal{H}^1$ such that $m_2$ and $m_3$ are colinear in $H^1$, there exists $\varphi \in [0, 2 \pi]$ and a lifting $\theta \in \mathcal{C}^{\frac{1}{2}} (\mathbb{R})$ such that
    \begin{equation}\label{eq:def_lift}
        \boldsymbol{m}(x)=R_\varphi \begin{pmatrix}
        \sin\theta(x)\\ \cos \theta(x)\\ 0
        \end{pmatrix}.
    \end{equation}
    If $\boldsymbol{m} \in \mathcal{H}^2$, then $\theta \in \mathcal{C}^{1, \frac{1}{2}} (\mathbb{R})$.
\end{lemma}

\begin{proof}
    The proof follows from easy arguments presented in~\cite{book_carbou,carbou2018stabilization}.
\end{proof}

\begin{lemma} \label{lem:steadystate_planar}
    For any steady-state $\boldsymbol{m} \in \mathcal{H}^1$, there exists $\varphi \in [0, 2 \pi]$ and a lifting $\theta \in \mathcal{C}^1 (\mathbb{R})$ such that \eqref{eq:def_lift} holds.
\end{lemma}

\begin{proof}
    The function $\tilde{\boldsymbol{m}}$ as defined in Lemma \ref{lem:chg_var_m} satisfies \eqref{eq:ODE_tilde_m}. 
    Let us show that it implies
    \begin{equation} \label{eq:col_expr}
        \tilde{m}_2 \partial_y \tilde{m}_3 - \tilde{m}_3 \partial_y \tilde{m}_2 = 0 \quad \text{on } \mathbb{R}.
    \end{equation}
    Indeed, using \eqref{eq:ODE_tilde_m}, there holds
    \begin{align*}
        \partial_y \Bigl( \tilde{m}_2 \partial_y \tilde{m}_3 - \tilde{m}_3 \partial_y \tilde{m}_2 \Bigr) &= \tilde{m}_2 \partial_{yy} \tilde{m}_3 - \tilde{m}_3 \partial_{yy}  m_2 \\
            &= \tilde{m}_2 ( \partial_{yy} \tilde{m}_3 + \sigma^2 (y) \tilde{m}_3 ) - \tilde{m}_3 ( \partial_{yy} \tilde{m}_2+ \sigma^2 (y) \tilde{m}_2 ) \\
            &= \tilde{m}_2 ( \tilde{\Lambda} (y) \tilde{m}_3 ) - \tilde{m}_3 ( \tilde{\Lambda} (y) \tilde{m}_2 ) \\
            &= 0,
    \end{align*}
    which shows that $\tilde{m}_2 \partial_y \tilde{m}_3 - \tilde{m}_3 \partial_y \tilde{m}_2$ is constant on $\mathbb{R}$.
    On the other hand, since $\tilde{\boldsymbol{m}} \in \mathcal{H}^1$, we have $\tilde{m}_2, \tilde{m}_3 \in H^1$, and thus $\tilde{m}_2 \partial_x \tilde{m}_3 - \tilde{m}_3 \partial_x \tilde{m}_2 \in L^1$.
    This proves \eqref{eq:col_expr}, which implies that $(\tilde{m}_2 (y), \partial_x \tilde{m}_2 (y))$ and $(\tilde{m}_3 (y), \partial_x \tilde{m}_3 (y))$ are colinear in $\mathbb R^2$ for every $y$ in $\mathbb R$.
    On the other hand, by \eqref{eq:ODE_tilde_m}, these two vector fields solve the same first order linear ODE system in $(u, v)$:
    \begin{equation*}
        \partial_x u = v, \quad \partial_x v = (1 + \tilde{\Lambda} (y)) u.
    \end{equation*}
    By uniqueness in the Cauchy-Lipschitz theorem, the colinearity factor is constant in $y$, and thus
    $m_2$ and $m_3$ are colinear. We can then use Lemma \ref{lem:lift_planar} to lift $\tilde{\boldsymbol{m}}$.
    The conclusion for $\boldsymbol{m}$ follows by defining $\theta = \tilde{\theta} \circ y$.
\end{proof}

\subsection{Energy of the lifting}

We introduce the energy for the lifting $E_s (\theta)$ defined as
\begin{equation}\label{expr:SS_infinite:Energy}
    E_s(\theta):=\frac{1}{2}\int_\R\theta'(x)^2s(x)\, \diff x
    +\frac{1}{2}\int_\R\cos^2\theta(x)\,s(x)\, \diff x.
\end{equation}
Moreover, we define the following spaces for lifting:
\begin{gather}
    {\mathscr W}:=\Big\{\theta:\R\to\R\;:\;\theta'\in L^2,\quad\cos\theta\in L^2\Big\}.
\end{gather}
The relation between the two energies is given in the following lemma.

\begin{lemma} \label{lem:rel_energies}
The energy $E_s$ is preserved by rotations with constant angle: more precisely, let $\varphi \in [0, 2 \pi]$, $R_\varphi$ the rotation matrix given by \eqref{expr:SS_infinite} and $\theta \in \mathcal{C}^1 (\mathbb{R})$ such that
    \begin{equation*} %\label{eq:def_lift_0}
        \boldsymbol{m}=R_\varphi \begin{pmatrix}
        \sin\theta\\ \cos \theta\\ 0
        \end{pmatrix} \in \mathcal{H}^1.
    \end{equation*}
    Then $\theta \in \mathscr W$ and $E_s (\boldsymbol{m}) = E_s (\theta)$.
\end{lemma}

This relation explains why we took the same notation for both energies, as there will be no confusion between them.
Furthermore, the limits of $\theta$ at $\pm \infty$ are related to the limits of $\boldsymbol{m}$ at $\pm \infty$. For instance, if $\lim_{\pm \infty} \boldsymbol{m} = \pm \mathbf{e_1}$, then there exists $k_\pm \in \Z$ such that
\begin{equation*}
    \lim_{\pm \infty} \theta = \pm \frac{\pi}{2} + 2 \pi k_\pm.
\end{equation*}
Since $\theta$ is defined up to a $2 \pi \mathbb{Z}$ additive constant, we can assume that $k_- = 0$, so that $\lim_{- \infty} \theta = - \frac{\pi}{2}$.
Moreover, $\theta$ can be changed in $- \pi - \theta$ by a rotation around $\mathbf{e_1}$ of angle $\pi$, so we can also assume that $k_+ \in \N$.
This is why we also define the following functional space for lifting:
\begin{equation*}
    {\mathscr W}_{\neq} \coloneqq \{ \theta \in \mathscr W \; : \; \lim_{- \infty} \theta = - \frac{\pi}{2} \text{ and } \lim_{+ \infty} \theta \in \frac{\pi}{2} + 2 \pi \mathbb{N} \}.
\end{equation*}

From the previous properties, the minimization of the energy $E_s (\boldsymbol{m})$ is related to the minimization of the energy of the lifting $E_s (\theta)$.

\begin{lemma}
    If $\boldsymbol{m}_0$ is a critical point of $E_s (\boldsymbol{m})$ in $\mathcal{H}^1_{\neq}$, then there exists $\theta_0 \in \mathscr W_{\neq}$ and $\varphi_0 \in \R$ such that
    \begin{itemize}
        \item $\theta_0$ is a lifting of $\boldsymbol{m}_0$:
        \begin{equation} \label{eq:def_lift_0}
            \boldsymbol{m}_0=R_{\varphi_0} \begin{pmatrix}
            \sin\theta_0\\ \cos \theta_0\\ 0
            \end{pmatrix},
        \end{equation}
        \item $\theta_0$ is a critical point of $E_s (\theta)$ in $\mathscr W_{\neq}$,
        \item If $\boldsymbol{m}_0$ is a minimizer of $E_s (\boldsymbol{m})$, then $\theta_0$ is a minimizer of $E_s (\theta)$.
    \end{itemize}
    Conversely, if $\theta_0$ minimizes $E_s (\theta)$ in ${\mathscr W}_{\neq}$, then, for any $\varphi_0 \in \mathbb{R}$, $\boldsymbol{m}_0$ defined by \eqref{eq:def_lift_0} is a minimizer of $E_s (\boldsymbol{m})$ in $\mathcal{H}^1_{\neq}$.
\end{lemma}

\begin{proof}
    The conclusions of Lemma \ref{lem:steadystate_planar} hold for any critical point (which is a steady state for \eqref{eq:LLG}), and we can assume $\theta_0 \in \mathscr{W}_{\neq}$ with the discussion above.
    Moreover, it is easy to verify that this $\theta_0$ is a critical point of $E_s (\theta)$ in $\mathscr{W}_{\neq}$, or even a minimizer if $\boldsymbol{m}_0$ is a minimizer for $E_s (\boldsymbol{m})$. 
    Conversely, a minimizer $\theta_0$ of $E_s (\theta)$ would be a minimizer of $E_s (\boldsymbol{m})$ for $\boldsymbol{m}$ which can be lifted as \eqref{eq:def_lift}. 
    We show now that, for any $\boldsymbol{m} \in \mathcal{H}^1_{\neq}$, we can find some $\tilde{\boldsymbol{m}} \in \mathcal{H}^1_{\neq}$ such that $\tilde{m}_3 = 0$ and $E_s (\tilde{\boldsymbol{m}}) \leq E_s (\boldsymbol{m})$. Such a $\tilde{\boldsymbol{m}}$ can be lifted by some $\tilde \theta \in \mathscr{W}_{\neq}$, so that $E_s (\tilde{\boldsymbol{m}}) = E_s (\tilde \theta) \geq E_s (\theta_0)$, and thus $E_s (\boldsymbol{m}) \geq E_s (\boldsymbol{m}_0)$ where $\boldsymbol{m}_0$ is defined by \eqref{eq:def_lift_0}, which is a sufficient property for the second part of the result to hold true.

    With the different limits of $\boldsymbol{m}$ at $\pm \infty$ and its continuity (since $\mathcal{H}^1 \subset \mathcal{C}^{0, \frac{1}{2}}$, we can easily prove that there exists $x_0 \in \mathbb{R}$ such that
    \begin{itemize}
        \item $\boldsymbol{m} (x_0) \not \in \{ - \mathbf{e_1}, \mathbf{e_1} \}$,
        \item $\boldsymbol{m} (a) = - \mathbf{e_1}$ if $a \coloneqq \sup \{ x < x_0, \boldsymbol{m} (x) \not \in \{ - \mathbf{e_1}, \mathbf{e_1} \} \} > - \infty$,
        \item $\boldsymbol{m} (b) = \mathbf{e_1}$ if $b \coloneqq \inf \{ x > x_0, \boldsymbol{m} (x) \not \in \{ - \mathbf{e_1}, \mathbf{e_1} \} \} < + \infty$.
    \end{itemize}
    From the regularity of $\boldsymbol{m}$, we also know that $- \infty \leq a < x_0$ and $x_0 < b \leq + \infty$ and $m_1 (x) \in (- 1, 1)$ for all $x \in (a, b)$.
    So we can define $\theta \coloneqq \arcsin m_1 \in \mathcal{C}^{0, \frac{1}{2}} ((a, b), (- \frac{\pi}{2}, \frac{\pi}{2}))$.
    As $\abs{m}^2 = 1$, we know that $\Bigl( \frac{m_2}{\cos \theta} \Bigr)^2 + \Bigl( \frac{m_3}{\cos \theta} \Bigr)^2 = 1$.
    Since $m_2$, $m_3$ and $\cos \theta$ are $\mathcal{C}^{0,\frac{1}{2}}$ function on $(a, b)$ and the latter does not vanish on this interval, we can find $\varphi \in \mathcal{C}^{0,\frac{1}{2}} ((a, b), \mathbb{R})$ such that $\cos \varphi = \frac{m_2}{\cos \theta}$ and $\sin \varphi = \frac{m_3}{\cos \theta}$.
    In consequence, with this definition of $\theta$ and $\varphi$, there holds for all $x \in \mathbb{R}$
    \begin{equation*}
        \boldsymbol{m} (x) = 
            \begin{pmatrix}
                \sin \theta (x) \\
                \cos \theta (x) \, \cos \varphi (x) \\
                \cos \theta (x) \, \sin \varphi (x)
            \end{pmatrix}.
    \end{equation*}
    We also point out that for a.e. $x \in (a, b)$
    \begin{align*}
        \boldsymbol{m}' (x) &= \theta' (x)
            \begin{pmatrix}
                \cos \theta (x) \\
                - \sin \theta (x) \, \cos \varphi (x) \\
                - \sin \theta (x) \, \sin \varphi (x)
            \end{pmatrix}
            + \varphi' (x)
            \begin{pmatrix}
                0 \\
                - \cos \theta (x) \, \sin \varphi (x) \\
                \cos \theta (x) \, \cos \varphi (x)
            \end{pmatrix}, \\
        \abs{\boldsymbol{m}' (x)}^2 &= (\theta' (x))^2 + (\varphi' (x))^2 \cos^2 \theta (x) + 2 \theta' (x) \, \varphi' (x) \, 
            \begin{pmatrix}
                \cos \theta (x) \\
                - \sin \theta (x) \, \cos \varphi (x) \\
                - \sin \theta (x) \, \sin \varphi (x)
            \end{pmatrix}
            \cdot
            \begin{pmatrix}
                0 \\
                - \cos \theta (x) \, \sin \varphi (x) \\
                \cos \theta (x) \, \cos \varphi (x)
            \end{pmatrix} \\
            &= (\theta' (x))^2 + (\varphi' (x))^2 \cos^2 \theta (x).
    \end{align*}
    Therefore,
    \begin{align*}
        E_s (\boldsymbol{m}) &\geq \int_a^b (\abs{\boldsymbol{m}' (x)}^2 + m_2 (x)^2 + m_3 (x)^2) s(x) \diff x \\
            &\geq \int_a^b \Bigl( (\theta' (x))^2 + (\varphi' (x))^2 \cos^2 \theta (x) + \cos^2 \theta (x) \Bigr) s (x) \diff x.
    \end{align*}
    In particular, $\theta' \in L^2 ((a, b))$.
    Now, we define
    \begin{equation*}
        \tilde{\boldsymbol{m}} (x) =
        \begin{cases}
            \begin{pmatrix}
                \sin \theta (x) \\
                \cos \theta (x) \\
                0
            \end{pmatrix} & \text{ if } x \in (a, b) \\
            - \mathbf{e_1} & \text{ if } - \infty < x \leq a, \\
            \mathbf{e_1} & \text{ if } b \leq x < + \infty.
        \end{cases}
    \end{equation*}
    From the regularity of $\theta$ and the (possible) continuous junction(s) at $a$ and $b$, we can easily see that $\tilde{\boldsymbol{m}} \in \mathcal{H}^1_{\neq}$ and
    \begin{align*}
        E_s (\tilde{\boldsymbol{m}}) &= \int_a^b \Bigl( \theta' (x)^2 + \cos^2 \theta (x) \Bigr) s (x) \diff x \\
            &\leq \int_a^b \Bigl( \theta' (x)^2 + (\varphi' (x))^2 \cos^2 \theta (x) + \cos^2 \theta (x) \Bigr) s (x) \diff x \\
            &\leq E_s (\boldsymbol{m}).
    \end{align*}
    The conclusion holds.
\end{proof}

\subsection{Existence of a domain wall}

\subsubsection{Transforms of function}

We know that steady states $\boldsymbol{m} \in \mathcal{H}^1_{\neq}$ of \eqref{eq:LLG} are critical points of $E_s (\boldsymbol{m})$.
On the other hand, they are also planar (Lemma \ref{lem:lift_planar}) and can be lifted with some function $\theta \in \mathscr{W}_{\neq}$ such that $E_s (\boldsymbol{m}) = E_s (\theta)$.
Therefore any steady state is related to a lifting $\theta$ which is a critical point for $E_s (\theta)$ in $\mathscr{W}_{\neq}$.

Our problem has thus been reduced to finding the critical points of the energy $E_s (\theta)$ on $\mathscr{W}_{\neq}$, but we can reduce even more this set.
For this, we define the set $W$ defined by
\begin{equation*} %\label{eq:def_W_0}
    {\mathscr W}_0:=\Big\{\theta\in{\mathscr W}\;:\;\;\lim\limits_{x\to\pm\infty}\theta(x)=\pm\frac{\pi}{2} \Big\} \subset \mathscr{W}_{\neq}.
\end{equation*}
\begin{equation} \label{eq:def_W}
    W:=\Big\{\theta\in{\mathscr W}_{\neq} \;:\; \overline{\theta (\R)}\subseteq [-\pi/2,\pi/2] \Big\} \subset \mathscr{W}_{\neq}.
\end{equation}
We must point out that the limit at $\pm\infty$ of any function in $W$ is completely determined.

\begin{lemma}
    There holds $W \subset {\mathscr W}_0$.
\end{lemma}

\begin{proof}
    For any $\theta \in \mathscr{W}_{\neq}$, we know that $\lim_{+ \infty} \theta$ exists and satisfies $\lim_{+ \infty} \theta \geq \frac{\pi}{2}$. Thus, if $\overline{\theta (\R)}\subset [-\pi/2,\pi/2]$, the opposite inequality is obvious, which yields the equality.
\end{proof}

Several transforms which can decrease this functional are at our disposal.
In particular, the threshold function $T (x):=\max\{-\pi/2\,;\,\min\{x;\pi/2\}\}$, which is a $1$-Lipschitz function, constant outside $[-\pi/2,\pi/2]$, allows us to reduce our search of a minimizer in $W$.

\begin{lemma} \label{lem:min_in_W}
    For any $\theta \in \mathscr{W}_{\neq}$, we define the function $\theta_\natural \coloneqq T \circ \theta$. Then $\theta_\natural \in W$ and $E_s (\theta_\natural) \leq E_s (\theta)$. Moreover, if $\theta \not\in W$, this inequality is strict.
\end{lemma}

\begin{proof}
    Let $\theta \in{\mathscr W}_{\neq}$. If $\overline{\theta (\R)}\subset [-\pi/2,\pi/2]$, then $\theta_\natural = \theta$ and $\lim_{+ \infty} \theta = \frac{\pi}{2}$, and the properties are easy.
    Assume now that $\overline{\theta (\R)}\supsetneq[-\pi/2,\pi/2]$. Then,
    \begin{equation*}
        I_{\theta}:=\Big\{x\in\R\;:\;\theta (x)\notin\Big[-\frac{\pi}{2};\frac{\pi}{2}\Big]\Big\}
    \end{equation*}
    is non-empty.
    We recall that $\theta$ is an absolutely continuous function as a consequence of the definition of ${\mathscr W}_{\neq}$. 
    Since there exists $x \in \R$ such that $\theta (x) \in (-\pi/2,\pi/2)$ for some value of $x$, this function cannot be identically constant on the set $I_\theta$. 
    Thus,
    \begin{equation}\label{Odin}
        \int_{I_{\theta}}\big|\theta'(x)\big|^2\diff x>0.
    \end{equation}
    By definition of $I_{\theta}$ and $T$, $\theta_\natural$ is constant on $I_\theta$. On the other hand, on $\R\setminus I_{\theta}$, $T\circ\theta = \theta$. Therefore
    \begin{align*}
        \int_\R \abs{\theta_\natural'(x)}^2s(x) \diff x &= \int_{\R\setminus I_{\theta}} \abs{\theta_\natural'(x)}^2 s(x) \diff x \\
            &= \int_{\R\setminus I_{\theta}} \abs{(T\circ\theta)'(x)}^2 s(x) \diff x \\
            &= \int_{\R\setminus I_{\theta}} \abs{\theta'(x)}^2s(x) \diff x.
    \end{align*}
    Combined with~\eqref{Odin}, we get
    \begin{equation}\label{Thor}
        \int_\R \abs{\theta_\natural'(x)}^2s(x) \diff x < \int_{\R} \abs{\theta'(x)}^2s(x) \diff x.
    \end{equation}
    Moreover, since the function $x\mapsto \cos^2(x)$ is minimal at $x=\pm\pi/2$, then by property of the function $T$:
    \begin{equation}\label{Loki}
        \int_\R \cos^2 \big(T \circ \theta (x) \big) s(x) \diff x \leq \int_\R \cos^2 \big( \theta (x) \big) s(x) \diff x.
    \end{equation}
    Therefore,~\eqref{Thor} and~\eqref{Loki} imply that $f\circ\theta_1\in{\mathscr W}$ and 
    \begin{equation*}
        E_s\big(f\circ\theta_1\big) < E_s\big(\theta_1).
    \end{equation*}
    Last, the property of $T$ shows that $\theta_\natural (\R) \subset [-\pi/2,\pi/2]$ and $\lim_{\pm \infty} \theta = \pm \frac{\pi}{2}$. Thus, $\overline{\theta_\natural (\R)} = [-\pi/2,\pi/2]$.
\end{proof}

Therefore, we can only consider functions $\theta \in W$.
Let us define the following quantity for any of these functions:
\begin{equation}\label{Gérard de Nerval}
    \rho[\theta]:= \inf\,\big\{x\in\R\;:\;\theta(x) \geq 0\big\}.
\end{equation}
Due to their opposite limits at infinity and their continuity, we know that $\rho[\theta] \in \mathbb{R}$.
It corresponds to the first time where $\theta$ vanishes.
We can therefore reduce our problem to monotonous functions. This is done in two steps.

\begin{lemma} \label{lem:transf_tilde_dag}
    For any function $\theta \in W$, let 
    \begin{equation}\label{Jean-Jacques Rousseau}
        \widetilde{\theta}(x):=\left\{\begin{array}{ll}
        \theta(x) \leq 0&\quad\text{if }x\leq\rho[\theta]\\
        |\theta(x)| \geq 0&\quad\text{otherwise, }
        \end{array}\right.
    \end{equation}
    and the \emph{best non-decreasing upper bound}
    \begin{equation*}
        \theta^\dag(x):=\left\{\begin{array}{ll}
        \inf_{y\in[x,+\infty)}\widetilde{\theta}(y)&\quad\text{if }x\leq\rho[\theta]\\
        \sup_{y\in(-\infty,x]}\widetilde{\theta}(y)&\quad\text{otherwise,}
        \end{array}\right.
    \end{equation*}
    where $\rho[\theta]$ is defined at~\eqref{Gérard de Nerval}.
    Then $\widetilde{\theta}, \theta^\dag \in W$, $\theta^\dag$ is non-decreasing and $E_s (\theta^\dag) \leq E_s (\widetilde{\theta}) \leq E_s (\theta)$. Moreover,
    \begin{itemize}
        \item if $\theta$ is non-decreasing, then $\widetilde{\theta} = \theta$,
        \item if $\widetilde{\theta}$ is non-decreasing, then $\theta^\dag = \widetilde{\theta}$.
    \end{itemize}
\end{lemma}

\begin{proof}
To start-with one can check directly that $\|\theta\|_{L^\infty}=\|\widetilde{\theta}\|_{L^\infty}=\|\theta^\dag\|_{L^\infty}$

\noindent\underline{Step 1}:
    Since the function $\cos^2$ is even, we have $\cos^2 y=\cos^2|y|$ for any $y \in \mathbb R$. Thus,
    \begin{equation}\label{Laurel}
        \int_\R\frac{\cos^2\theta(x)}{2}\,s(x)\diff x\;=\;\int_\R\frac{\cos^2\widetilde{\theta}(x)}{2}\,s(x)\diff x.
    \end{equation}
    On the other hand, the triangle inequality gives that $\big| |f(x+y)|-|f(x)|\big|\leq|f(x+y)-f(x)|,$ which implies 
    \begin{equation}\label{Gabrielle d'Estrées}
        \int\big||f|'(x)\big|^2  s(x) \diff x\;\leq\;\int\big|f'(x)\big|^2  s(x) \diff x.
    \end{equation}
    We now observe that since $\rho[\theta]$ is a zero of $\theta$, then the transformation giving $\widetilde{\theta}$ from $\theta$ does not create a discontinuity at the point $\rho[\theta]$. Therefore the function $\widetilde{\theta}$ is weakly differentiable if and only if $\theta$ is weakly differentiable. Its weak derivative satisfy for almost every $x\in\R$:
    \begin{equation}\label{Henri IV}
        \big|\widetilde{\theta}'(x)\big|\;=\;\big|\theta'(x)\big|
    \end{equation}
    If we now combine~\eqref{Gabrielle d'Estrées} and~\eqref{Henri IV}, we get:
    \begin{equation}\label{Hardy}
        \int_\R\big|\theta'(x)\big|^2 s(x) \diff x\geq\int_\R\big|\widetilde{\theta}\,'(x)\big|^2 s(x) \diff x.
    \end{equation}
    Thus,~\eqref{Laurel} and~\eqref{Hardy} together imply
    \begin{equation} \label{eq:ineq_energy_tilde_theta}
        E_s\big(\,\widetilde{\theta}\,\big)\leq E_s(\theta). 
    \end{equation}
    In particular this implies $\widetilde{\theta}\in W$.
    
\noindent\underline{Step 2}:
    The function $\theta^\dag$ is indeed non-decreasing.
    One can check that the set of $x$ such that $\theta^\dag(x)\neq\widetilde{\theta}(x)$ is a reunion of 2-by-2 disjoint intervals and on each of these intervals the function  $\theta^\dag$ is constant (for more details and properties, see~\cite{Godard-Cadillac_2021_Tamping}). As a consequence, we have:
    \begin{equation*}%\label{Ragnarok}
        \int_\R\big|\big(\theta^\dag\big)'(x)\big|^2s(x) \diff x\;\leq\;\int_\R\big|\widetilde{\theta}'(x)\big|^2s(x) \diff x,
    \end{equation*}
    with equality if and only if $\widetilde{\theta}\equiv\theta^\dag$. 
    
    From the properties of $W$, there holds
    \begin{equation*}%\label{Remus}
        \forall\,x\leq\rho[\theta],\quad-\frac{\pi}{2}\leq\theta^\dag(x)\leq\widetilde{\theta}(x)\leq 0.
    \end{equation*}
    Thus,
    \begin{equation*}
        \int_{-\infty}^{\rho[\theta]}\cos^2\big(\theta^\dag(x)\big)\,s(x) \diff x\;\leq\;\int_{-\infty}^{\rho[\theta]}\cos^2\big(\widetilde{\theta}(x)\big)\,s(x) \diff x\;
    \end{equation*}
    Similarly,
    \begin{equation*}%\label{Romulus}
        \forall\,x\geq\rho[\theta],\quad\frac{\pi}{2}\geq\theta^\dag(x)\geq\widetilde{\theta}(x)\geq 0.
    \end{equation*}
    Thus,
    \begin{equation*}
        \int_{\rho[\theta]}^{\infty}\cos^2\big(\theta^\dag(x)\big)\,s(x) \diff x\;\leq\;\int_{\rho[\theta]}^{\infty}\cos^2\big(\widetilde{\theta}(x)\big)\,s(x) \diff x\;
    \end{equation*}
    Combining these inequalities gives:
    \begin{equation*}
        \int_\R\cos^2\big(\theta^\dag(x)\big)\,s(x) \diff x\;\leq \int_\R\cos^2\big(\widetilde{\theta}(x)\big)\,s(x) \diff x\;
    \end{equation*}
    Similarly as before, this gives an inequality of the energy for this transform. Combining it with \eqref{eq:ineq_energy_tilde_theta} yields the conclusion.

\end{proof}

Last, if the notch is symmetric, there is also another transform using symmetric decreasing rearrangement (see~\cite{Lieb_Loss_2014} for details about rearrangements) that we can use.

\begin{lemma} \label{lem:rearrange_odd}
    Assume that $s \in \mathscr{A}_a (\Omega)$ and $s(\cdot) \neq 1$. For $\theta \in W$ non-decreasing, we define
    \begin{equation} \label{eq:change_var_theta}
        \vartheta = \theta \circ y^{-1},
    \end{equation}
    where $y$ is defined in \eqref{eq:def_y_change_var}, and
    \begin{equation*}
        \vartheta^\ddag(y) \coloneqq \bigg(\frac{\pi}{2}-|\vartheta|\bigg)^\sharp(y)-\frac{\pi}{2}
        \qquad \text{and} \qquad
        \theta^\ast(x):=\widetilde{\,\vartheta^\ddag\,} (y),
    \end{equation*}
    where $\sharp$ denotes the symmetric decreasing rearrangement of non-negative functions on $\R$ and $\sim$ denotes the transform of Lemma \ref{lem:transf_tilde_dag}.
    Then $\theta^\ast$ is a non-decreasing odd function in $W$, and $E_s (\theta^\ast) \leq E_s (\theta)$. The inequality is strict if $\theta$ is not odd.
\end{lemma}

\begin{proof}
    By the Pólya–Szegő inequality for the rearrangements, we have
    \begin{equation}\label{Le prince d'Aquitaine}
        \int_\R\big(\vartheta^\ddag\big)'(y)^2 \diff y\;\leq\;\int_\R\vartheta'(y)^2 \diff y.
    \end{equation}
    Since the function $\theta$ is assumed to be non-decreasing, so is $\vartheta$.
    Using the case of equality in the Pólya–Szegő inequality, we infer that the inequality above is an equality if, and only if, 
    \begin{equation}\label{La tour abolie}
        \exists\;a\in\R,\quad \forall y\in \R, \quad -\big|\vartheta(y-a)\big|=\vartheta^\ddag(y).
    \end{equation}
    On the other hand, since the function $x\mapsto s(x)$ belongs to ${\mathscr A}_a(\R),$ then we have $-s(x)=(-s)^\sharp(x)$. 
    It is direct to check that this property is then also satisfied for $\sigma$.
    We now recall that if $f$ is a non-negative non-decreasing function then we have $f\circ(g^\sharp)=(f\circ g)^\sharp$ almost everywhere. Thus, almost everywhere:
    \begin{equation*}
        \cos^2\big(\vartheta^\ddag\big)=\sin^2\bigg(\bigg(\frac{\pi}{2}-|\vartheta|\bigg)^\sharp\bigg)=\bigg(\sin^2\bigg(\frac{\pi}{2}-|\vartheta|\bigg)\bigg)^\sharp
    \end{equation*}
    and then $\cos^2\big(\vartheta^\ddag\big)=\cos^2\big(\vartheta\big)^\sharp.$
    Therefore, by the Hardy-Littlewood rearrangement inequality,  we have
    \begin{equation}\label{Le ténébreux}
        \int_\R\cos^2\big(\vartheta(y)\big)\,(-\sigma^2)(y) \diff y\;\leq\;\int_\R\big(\cos^2\big(\vartheta\big)\big)^\sharp(y)\,(-\sigma^2)^\sharp(y) \diff y\;=\;\int_\R\cos^2\big(\vartheta^\ddag(y)\big)\,(-\sigma^2)(y) \diff y.
    \end{equation}
    We combine~\eqref{Le prince d'Aquitaine} and~\eqref{Le ténébreux} and get
    \begin{equation}\label{la treille}
        E_s(\theta^\ast)\;\leq\;E_s(\theta),
    \end{equation}
    and the case of equality is given by~\eqref{La tour abolie}. 
    
    To conclude that $\theta^\ast$ is an odd function, there remain to prove that the number $a\in\R$ appearing in~\eqref{La tour abolie} is actually $0$. We consider a function $\vartheta$ such that~\eqref{La tour abolie} holds for some value of $a>0$ and we prove that in this case~\eqref{Le ténébreux} is actually a strict inequality. The same reasoning also holds in the case $a<0$. First, using the parity of $\cos^2$ and~\eqref{La tour abolie}, we write 
    \begin{equation*}\begin{split}
       I(\vartheta)&:=\frac{1}{2}\int_\R\cos^2\big(\vartheta(y)\big)\,\sigma^2(y) \diff y=\frac{1}{2}\int_\R\cos^2\big(\vartheta^\ddag(y+a)\big)\,\sigma^2(y) \diff y.\\
        &=\int_\R\bigg(\int_{-\frac{\pi}{2}}^{\vartheta^\ddag(y+a)}\cos(\nu)\,\sin(\nu) \diff \nu\bigg)\bigg(\int_0^{\sigma^2(y)}d\mu\bigg) \diff y\\
        &=\int_\R\int_{-\frac{\pi}{2}}^{0}\int_0^{+\infty}\mathbbm{1}_{\big\{\vartheta^\ddag(y+a)\geq\nu\big\}}\,\mathbbm{1}_{\big\{\sigma^2(y)\geq\mu\big\}}\,\cos(\nu)\,\sin(\nu) \diff \mu \diff \nu \diff y,
    \end{split}
    \end{equation*}
    where for the last inequality we used the fact that $\vartheta^\ddag\leq 0$.
    By the Fubini theorem,
    \begin{equation*}
        I(\vartheta)=\int_{-\frac{\pi}{2}}^{0}\int_0^{+\infty}\meas\Big(\big\{y\in\R\,:\,\vartheta^\ddag(y+a)\geq\nu\quad\text{and}\quad\sigma^2(y)\geq\mu\big\}\Big)\,\cos(\nu)\,\sin(\nu) \diff \nu \diff \mu.
    \end{equation*}
    Since the function $y\mapsto-\sigma^2(y)$ is symmetric decreasing, we have
    \begin{equation*}
        \{y\in\R\,:\sigma^2(y)\leq\mu\big\}=[-m_\mu,\,m_\mu],
    \end{equation*}
    where $m_\mu:=\meas\{y\in\R\,:\sigma^2(y)\leq\mu\big\}/2$. Similarly, since the function $y\mapsto\vartheta^\ddag(y)$ is symmetric decreasing, we have
    \begin{equation*}
        \{y\in\R\,:\vartheta^\ddag(y+a)\geq\nu\big\}=\big[-q_\nu-a,\,q_\nu-a\big],
    \end{equation*}
    where $q_\nu:=\meas\{y\in\R\,:\vartheta^\ddag(y)\leq\nu\big\}/2$. Thus,
    \begin{equation}\label{le pampre}
        I(\vartheta)=\int_{-\frac{\pi}{2}}^{+\infty}\int_0^{+\infty}\meas\Big(\big[-q_\nu-a,\,q_\nu-a\big]\setminus[-m_\mu,\,m_\mu]\Big)\,\cos(\nu)\,\sin(\nu) \diff \nu \diff \mu.
    \end{equation}
    Similar computations lead to
    \begin{equation}\label{la rose}
        I(\vartheta^\ddag)=\int_{-\frac{\pi}{2}}^{+\infty}\int_0^{+\infty}\meas\Big(\big[-q_\nu,\,q_\nu\big]\setminus[-m_\mu,\,m_\mu]\Big)\,\cos(\nu)\,\sin(\nu) \diff \nu \diff \mu.
    \end{equation}
    It is a direct consequence of the Hardy-Littlewood rearrangement inequality to have:
    \begin{equation}\label{L'inconsolé}
        \meas\Big(\big[-q_\nu-a,\,q_\nu-a\big]\setminus[-m_\mu,\,m_\mu]\Big)\;\geq\;\meas\Big(\big[-q_\nu,\,q_\nu\big]\setminus[-m_\mu,\,m_\mu]\Big)
    \end{equation}
    Therefore, we get the announced strict inequality if and only if the inequality above is strict for a set of $\nu$ that have a non-vanishing Lebesgue measure. More precisely, we study
    \begin{equation*}
        R_\mu:=\Big\{\nu\in\Big[-\frac{\pi}{2},0\Big]\;:\;\text{Inequality}~\eqref{L'inconsolé}\text{ is strict}.\,\Big\}.
    \end{equation*}
    We first remark that since $x\mapsto s(x)$ is not constant, so is $x\mapsto\sigma(x)$ and thus
    \begin{equation*}
        M:=\Big\{\mu\geq0\,:\,0<m_\mu<+\infty\Big\}
    \end{equation*}
    is an interval with non-empty interior. Moreover, we have
    \begin{equation*}
       \Big\{y\in\R\,:\vartheta^\ddag(y)\geq0\Big\}=\{0\},\qquad\text{and}\qquad \Big\{y\in\R\,:\vartheta^\ddag(y)\geq-\frac{\pi}{2}\Big\}=\R.
    \end{equation*}
    Then, since the function $\vartheta$ is continuous increasing, we have for all $p>0$ a unique $\nu$ such that $q_\nu=p$. 
    Therefore, assuming for instance that\footnote{In the case $a<0$, we can do the same reasoning by simply swapping signs.} $a>0$,
    \begin{equation*}
        \forall\;\mu\in M,\qquad\Big\{\nu\in\Big[-\frac{\pi}{2},\;0\Big]\;:\;m_\mu-a<q_\nu<m_\mu+a\Big\}\quad\text{is a non-empty interval}.
    \end{equation*}
    We now observe that~\eqref{L'inconsolé} is a strict inequality if and only if $q_\nu-a< m_\mu$ or $-q_\nu-a<-m_\mu$. It is in particular true when $m_\mu-a<q_\nu<m_\mu+a$.
    Thus,
    \begin{equation*}
        \forall\;\mu\in M,\qquad\meas\,R_\mu>0.
    \end{equation*}
    Plugging this back into~\eqref{le pampre} and~\eqref{la rose} gives that~\eqref{Le ténébreux} is a strict inequality.
    Thus, we proved that Inequality~\eqref{la treille} is an equality if and only if $-|\vartheta|=\vartheta^\ddag$. This property is equivalent to $\vartheta$ odd, and thus $\theta$ odd. This concludes the proof.
\end{proof}

\subsubsection{The switch must be performed inside the notch}

The last transform is a simple but still powerful translation in order to localize the switch in (or near enough) the notch. This property will be very important, in particular in order to get a compactness property for the minimizing sequence for the problem.

\begin{lemma}[Localization of the spin switch] \label{lem:localization_spin_switch}
Let $\theta\in W$ non-decreasing. Assume that the notch profile $s$ is equal to $1$ outside some interval $[a,b]$ and is below $1$ inside this interval. Then, there exists a non-decreasing function $\theta_0\in W$ such that:\medskip

$(i)$ The set of zeros of $\theta_0$ intersects the interval $[a,b]$.\medskip

$(ii)$ We have $E_s(\theta)\geq E_s(\theta_0)$.
\end{lemma}

\begin{proof} 
In the case where the sets of zeros for $\theta$ intersects the interval $[a,b]$, it is enough to define $\theta_0\equiv\theta$. Otherwise, we proceed to a translation of the function $\theta$ until this property is satisfied.
We first observe that the set of zeros $\{x\in\R:\theta(x)=0\}$ is an interval $[x_0,x_1]$ since $\theta$ is non-decreasing.
Assume now, without loss of generality, that $x_1<a$
 we consider again the change of variable $\theta\leftrightarrow\vartheta$ given by~\eqref{eq:change_var_theta} so that we can work with the following formulation of the energy:
    \begin{equation*}
        E_s(\theta)\;=\;\frac{1}{2}\int_\R\vartheta'(y)^2\,dy+\frac{1}{2}\int_\R\cos^2\big(\vartheta(y)\big)\,\sigma^2(y)\,dy.
    \end{equation*}
Since the change of variable is an increasing diffeomorphism of $\R$, then $\sigma$ is equal to $1$ outside $[\mathfrak a_-, \mathfrak a_+]$. We also have that the largest zero of $\vartheta$, which we denote $x_1'$  (and that is the image of $x_1$ by the change of variable) is such that $x_1'<a'$. We define 
    \begin{equation*}
        \vartheta_0(x):=\vartheta(x-\delta)
    \end{equation*}
    where $\delta = a'-x_1>0$ and we define the function $\theta_0$ from $\vartheta_0$ using the change of variable formula~\eqref{eq:change_var_theta}.
    To compute the difference of the energies between $\theta$ and $\theta_0$, we first observe that
    \begin{equation*}
        \int_\R\vartheta'(y)^2dy=\int_\R\vartheta_0'(y)^2dy.
    \end{equation*}
    On the other hand, one has:
    \begin{equation*}
        \cos^2\big(\vartheta(y)\big)\,=\,\sin^2\bigg(\frac{\pi}{2}-\big|\vartheta(y)\big|\bigg)\,=\,2\int_0^{\frac{\pi}{2}-|\vartheta(y)|}\cos\nu\sin\nu\,d\nu\,=\,\int_0^{\frac{\pi}{2}-|\vartheta(y)|}\sin(2\nu)\,d\nu
    \end{equation*}
    Therefore with the Fubini theorem, we are lead to
    \begin{equation}\label{super plug}
        \int_\R\cos^2\big(\vartheta(y)\big)\,\sigma^2(y)\,dy=\int_\R\int_0^{+\infty}\mathbbm{1}_{\big\{\frac{\pi}{2}-|\vartheta(y)|>\nu\big\}}\,\sin(2\nu)\,\sigma^2(y)\,d\nu\,dy
    \end{equation}
    The same computation holds for $\vartheta_0$. We observe that since $\theta$ and then $\vartheta$ is increasing, then the function $\frac{\pi}{2}-|\vartheta|$ is unimodal (i.e non increasing on $(-\infty,x_1']$ and non decreasing on $[x_1', +\infty)$).
    As a consequence, the super-level sets $\big\{\frac{\pi}{2}-|\vartheta(y)|>\nu\big\}$ are intervals and they satisfy:
    \begin{equation}\label{bleue}
        y_\nu(\vartheta):= \inf\bigg\{\frac{\pi}{2}-|\vartheta(y)|>\nu\bigg\}\leq x_1' \leq\sup\bigg\{\frac{\pi}{2}-|\vartheta(y)|>\nu\bigg\} =: z_\nu(\vartheta)
    \end{equation}
    We also have
    \begin{equation}\label{verte}
         y_\nu(\vartheta)\leq y_\nu(\vartheta)+\delta= y_\nu(\vartheta_0),\qquad\text{and}\qquad z_\nu(\vartheta)\leq z_\nu(\vartheta)+\delta = z_\nu(\vartheta_0).
    \end{equation}
    We now compute
    \begin{equation}\label{plug}\begin{split}
        \int_\R&\mathbbm{1}_{\big\{\frac{\pi}{2}-|\vartheta(y)|>\nu\big\}}\,\sigma^2(y)\,dy-\int_\R\mathbbm{1}_{\big\{\frac{\pi}{2}-|\vartheta_0(y)|>\nu\big\}}\,\sigma^2(y)\,d\nu\,dy\\
        &=\int_{y_\nu(\vartheta)}^{y_\nu(\vartheta)+\delta}\sigma^2(y)\,dy-\int_{z_\nu(\vartheta)}^{z_\nu(\vartheta)+\delta}\sigma^2(y)\,dy
        \end{split}
    \end{equation}
    We now observe that, combing~\eqref{bleue} and~\eqref{verte}:
    \begin{equation*}
        y_\nu(\vartheta)\leq y_\nu(\vartheta)+\delta\leq a'.
    \end{equation*}
    Therefore, $\sigma^2$ is identically equal to $1$ on the interval of integration $[y_\nu(\vartheta);{y_\nu(\vartheta)+\delta}].$ Since $1$ is the maximal value of $\sigma^2$ we conclude that 
    \begin{equation*}%\label{rouge}
        \int_{y_\nu(\vartheta)}^{y_\nu(\vartheta)+\delta}\sigma^2(y)\,dy\geq\int_{z_\nu(\vartheta)}^{z_\nu(\vartheta)+\delta}\sigma^2(y)\,dy
    \end{equation*}
    Plugging this back into~\eqref{plug} and then in~\eqref{super plug}, we eventually conclude:
    \begin{equation*}
        E_s(\theta)\geq E_s(\theta_0). \qedhere
    \end{equation*}

\end{proof}

\subsubsection{The main result}

We are now in the position to prove the existence of a minimizer for $E_s (\theta)$.

%%%%%
\begin{proposition}[Existence of a minimizer]\label{prop:strict monotony}~
Let $s \in \mathcal{U}_a (\mathbb{R})$ such that $s \not \equiv 1$. Then the following properties holds.

$(i)\;$ The energy $E_s$ introduced at~\eqref{expr:SS_infinite:Energy} admits a minimizer in $\mathscr{W}_{\neq}$.
In particular, Equation~\eqref{expr:SSTheta_infinite} admits a solution.\vspace{0.1cm}

$(ii)\;$ If $\theta_1\in{\mathscr W}_{\neq} \setminus W$, there exists $\theta_0\in W$ such that 
\begin{equation*}%\label{eq:decrease energy in W}
    E_s(\theta_0)<E_s(\theta_1).
\end{equation*}
In particular, any minimizer of $E_s$ belongs to $W$.

$(iii)$ Any $\theta$ minimizer of the energy $E_s$ in $W$ is an increasing function on $\R$.

$(iv)$ If $s \in \mathscr{A}_a (\Omega)$, then any minimizer is odd.
\end{proposition}

\begin{proof}
    First, we point out that, from the assumption that $s \in \mathcal{U}_a (\mathbb{R})$ and $s \not \equiv 1$, $s \equiv 1$ outside some interval $[a,b]$ and is strictly below $1$ inside this interval.
    The property $(ii)$ is proved by Lemma \ref{lem:min_in_W}.
    For $(iii)$, we first use Lemma \ref{lem:transf_tilde_dag} in order to show that any minimizer is non-decreasing.
    Invoking now the strong maximum principle for elliptic equations, we get that any solution to~\eqref{expr:SS_infinite:div a grad formulation} cannot be constant on a non trivial interval without being constant on the whole $\R$. This is eventually contradictory with $\theta\in W$ so that a minimizer of $E_s$ in $W$ is necessarily increasing.
    Moreover, Lemma \ref{lem:rearrange_odd} leads to $(iv)$.
    
    In order to prove $(i)$, let $\theta_n \in \mathscr W_{\neq}$ be a minimizing sequence, that is $E_s(\theta_n) \to \mathfrak m$ (defined in \eqref{eq:min_prob}). Using the transformations in Lemmas \ref{lem:min_in_W} and \ref{lem:transf_tilde_dag}, we can assume that every $\theta_n$ belongs to $W$ and is non-decreasing.
    Moreover, due to Lemma \ref{lem:localization_spin_switch}, we can also assume that, for any $n \in \N$, we have some $x_n \in [a, b]$ such that $\theta_n (x_n) = 0$.
    Since $\norm{\theta_n'}_{L^2}^2 \leq \frac{2}{s_0} E_s (\theta_n)$, $\theta_n'$ is uniformly bounded in $L^2 (\mathbb{R})$. Therefore there exists $\theta \in H^1_{\textnormal{loc}}$ such that (up to a sub-sequence) $\theta_n'$ converges weakly to $\theta'$ in $L^2$ (and in particular $\theta' \in L^2 (\mathbb{R})$) and $\theta_n \rightarrow \theta$ in $\mathcal{C} (K)$ for any compact $K$.
    Moreover, since $f (x, u) = u^2 s(x)$ is convex in $u$, there holds
    \begin{equation} \label{eq:est_d_theta_n_lim}
        \int (\theta' (x))^2 \, s(x) \diff x \leq \liminf_{n\to +\infty} \int (\theta_n' (x))^2 \, s(x) \diff x.
    \end{equation}
    We must also point out that, since $x_n \in [a, b]$ is uniformly bounded, we have some $\overline{x} \in [a, b]$ such that (up to a further sub-sequence) $x_n \rightarrow \overline{x}$. From the uniform convergence of $\theta_n$ on compact sets and the fact that $\theta_n (x_n) = 0$, there holds $\theta (\overline{x}) = 0$. Moreover, every $\theta_n$ is non-decreasing and satisfies $\abs{\theta_n} \leq \frac{\pi}{2}$, which implies that $\theta$ is also non-decreasing and that $\abs{\theta} \leq \frac{\pi}{2}$ due to the uniform convergence on compact sets.
    
    On the other hand, we also have a uniform bound of $\cos \theta_n$ in $H^1$. Indeed, $(\cos \theta_n)' = - \theta_n' \sin \theta_n$, so that
    \begin{equation*}
        \norm{\cos \theta_n}_{H^1}^2 \leq \norm{\cos \theta_n}_{L^2}^2 + \norm{\theta_n'}_{L^2}^2 \leq \frac{2}{s_0} E_s (\theta_n).
    \end{equation*}
    This estimate proves that there exists $\zeta \in H^1 (\mathbb{R})$ such that (also up to a further sub-sequence) $\cos \theta_n$ converges towards $\zeta$ weakly in $H^1$ and in particular $\cos \theta_n \rightarrow \zeta$ in $\mathcal{C} (K)$ for any compact $K$.
    But we also have $\cos \theta_n \rightarrow \cos \theta$ in $\mathcal{C} (K)$ for any compact $K$, since $\cos$ is continuous. Thus, $\zeta = \cos \theta$.
    Similar arguments as for $\theta'$ yields
    \begin{equation} \label{eq:est_cos_theta_n_lim}
        \int \cos^2 ( \theta (x) ) \, s(x) \diff x \leq \liminf_{n\to +\infty} \int \cos^2 ( \theta_n (x) ) \, s(x) \diff x.
    \end{equation}

    In particular, we get $\cos \theta \in H^1 (\mathbb{R})$, which proves that $\lim_{\abs{x} \rightarrow \infty} \cos \theta_n (x) = 0$. But $\theta$ is non-decreasing and satisfies $\theta (\overline{x}) = 0$ and $\abs{\theta} \leq \frac{\pi}{2}$, we have $- \frac{\pi}{2} \leq \theta (x) \leq 0$ for $x \leq \overline{x}$ and $0 \leq \theta (x) \geq \frac{\pi}{2}$ for $x \geq \overline{x}$. Thus, it is now easy to get that $\lim_{\pm \infty} \theta = \pm \frac{\pi}{2}$. From all these properties, we get $\theta \in W$, and from \eqref{eq:est_d_theta_n_lim} and \eqref{eq:est_cos_theta_n_lim}
    \begin{align*}
        E_s (\theta) &= \int \Bigl[ (\theta' (x))^2 + \cos^2 ( \theta (x) ) \Bigr] \, s(x) \diff x \\
            &\leq \liminf_{n\to +\infty} \int (\theta_n' (x))^2 \, s(x) \diff x + \liminf_{n\to +\infty} \int \cos^2 ( \theta_n (x) ) \, s(x) \diff x \\
            &\leq \liminf_{n\to +\infty} \int \Bigl[ (\theta_n' (x))^2 + \cos^2 ( \theta_n (x) ) \Bigr] \, s(x) \diff x \\
            &\leq \liminf_{n\to +\infty} E_s (\theta_n) = \mathfrak{m}.
    \end{align*}
    Thus, $\theta$ is a minimum for $E_s (\theta)$ and $\boldsymbol{m}$ defined as \eqref{eq:def_lift} for any $\phi$ is a minimum of $E_s (\boldsymbol{m})$.
\end{proof}

\section{Uniqueness of the domain wall}\label{sec:uniqueness}

%%%%%%%%%%%%%%%%%%%%%%%%%%%%%%%%%%%%%%%%%%%%%%%%%%%%%%%%%%%%%%
%%%%%%%%%%%%%%%%%%%%%%%%%%%%%%%%%%%%%%%%%%%%%%%%%%%%%%%%%%%%%%
\subsection{First properties}

\subsubsection{Energy point-wise inequality}

It is notable that $\theta$ solving \eqref{expr:SSTheta_infinite} satisfies a point-wise energy-like property.

\begin{lemma} \label{lem:energy_eq_theta}
    For any $\theta \in \mathscr W$ satisfying \eqref{expr:SSTheta_infinite}, there holds
    \begin{equation} \label{expr:energy_eq_theta}
        \theta'^2 (y)-\cos^2 (\theta (y)) \geq 0 \quad \text{for a.e. } y \in \R.
    \end{equation}
    This inequality is an equality if $\abs{y} > a$. There also holds $\theta'^2 -\cos^2 (\theta) \not\equiv 0$ if $s(\cdot) \neq 1$ and if $\theta$ is not constant.
\end{lemma}

\begin{proof}
    Let $\vartheta (y) = \theta (x)$ where $y$ is defined in \eqref{eq:def_y_change_var}. Then $\vartheta$ is a weak solution to
    \begin{equation} \label{eq:ODE_vartheta}
        \vartheta'' + \sigma^2 \cos \vartheta \sin \vartheta = 0.
    \end{equation}
    We also know that $\cos \vartheta \in L^2$, so that $\vartheta'' \in L^2$, and the previous equality is true in $L^2$. We can therefore multiply this equality by $\vartheta' \in L^2$, so that
    \begin{equation*}
        \frac{\diff}{\diff y} (\vartheta')^2 - \sigma^2 \frac{\diff}{\diff y} \cos^2 (\vartheta) = 0.
    \end{equation*}
    This equality is true in $L^1$. Thus, integrating it on $(y_0, y_1)$ for any $y_0, y_1 \in \mathbb{R}$ leads to
    \begin{equation*}
        (\vartheta')^2 (y_1) - \int_{y_0}^{y_1} \sigma^2 (y) \frac{\diff}{\diff y} \cos^2 (\vartheta (y)) \diff y = (\vartheta')^2 (y_0).
    \end{equation*}
    Moreover, there also holds
    \begin{equation*}
        \int_{y_0}^{y_1} \sigma^2 (y_1) \frac{\diff}{\diff y} \cos^2 (\vartheta (y)) \diff y = \sigma^2 (y_1) \Bigl( \cos^2 (\vartheta (y_1)) - \cos^2 (\vartheta (y_0)) \Bigr),
    \end{equation*}
    so that
    \begin{equation*}
        (\vartheta')^2 (y_1) - \sigma^2 (y_1) \cos^2 (\vartheta (y_1)) - \int_{y_0}^{y_1} (\sigma^2 (y) - \sigma^2 (y_1)) \frac{\diff}{\diff y} \cos^2 (\vartheta (y)) \diff y = (\vartheta')^2 (y_0) - \sigma^2 (y_1) \cos^2 (\vartheta (y_0)).
    \end{equation*}
    Using the fact that $\vartheta' \in H^1$ and $\cos \vartheta \in H^1$, we know that, for $y_1$ fixed, the right-hand side goes to $0$ when $\abs{y_0} \rightarrow \infty$. Thus we get, for any $y_1 \in \mathbb{R}$,
    \begin{align*}
        (\vartheta')^2 (y_1) - \sigma^2 (y_1) \cos^2 (\vartheta (y_1)) - \int_{- \infty}^{y_1} (\sigma^2 (y) - \sigma^2 (y_1)) \frac{\diff}{\diff y} \cos^2 (\vartheta (y)) \diff y &= 0, \\
        (\vartheta')^2 (y_1) - \sigma^2 (y_1) \cos^2 (\vartheta (y_1)) + \int_{y_1}^{\infty} (\sigma^2 (y) - \sigma^2 (y_1)) \frac{\diff}{\diff y} \cos^2 (\vartheta (y)) \diff y &= 0.
    \end{align*}
    If $y_1 \leq 0$, define $\sigma_{y_1}$ by
    \begin{equation*}
        \sigma_{y_1} (y) \coloneqq 
        \begin{cases}
            \sigma^2 (y_1) - \sigma^2 (y) & \text{if } y < y_1, \\
            0 & \text{if } y \geq y_1.
        \end{cases}
    \end{equation*}
    Recall that $\sigma$ is equal to $1$ outside an interval denoted $[\mathfrak a_-, \mathfrak a_+]$, according to the proof of Lemma~\ref{lem:localization_spin_switch}.  From the properties of $\sigma$, we know that $\sigma_{y_1}$ is a non-decreasing function with bounded variation, which is constant for $y \leq \mathfrak{a}_-$ and for $y \geq y_1$. Thus, its derivative is a (non-negative) measure supported on $[\mathfrak{a}_-, y_1]$, that we call $\diff \sigma_{y_1}$, and for any non-negative continuous function $f$, there holds
    \begin{equation*}
        \int_{\mathbb{R}} \sigma_{y_1} (y) f'(y) \diff y = - \int f (y) \diff \sigma_{y_1} (y) \leq 0.
    \end{equation*}
    This inequality becomes an equality if $f$ is supported outside of $[\mathfrak{a}_-, y_1]$ (and in particular if $y_1 < \mathfrak{a}_-$).
    Thus, since $\vartheta$ is continuous (and then so is $\cos^2 \vartheta$), there holds
    \begin{equation*}
        \int_{- \infty}^{y_1} (\sigma^2 (y) - \sigma^2 (y_1)) \frac{\diff}{\diff y} \cos^2 (\vartheta (y)) \diff y = - \int_{\mathbb{R}} \sigma_{y_1} (y) \frac{\diff}{\diff y} \cos^2 (\vartheta (y)) \diff y \geq 0,
    \end{equation*}
    and the last inequality is an equality if $y_1 < \mathfrak{a}_-$.
    The case $y_1 \geq 0$ can be treated similarly, so that, for all $y_1 \in \mathbb{R}$,
    \begin{align*}
        (\vartheta')^2 (y_1) - \sigma^2 (y_1) \cos^2 (\vartheta (y_1)) &\geq 0 \\
        (\vartheta')^2 (y_1) - \sigma^2 (y_1) \cos^2 (\vartheta (y_1)) &= 0 \qquad \text{if } y_1 \leq \mathfrak{a}_- \; \text{ or } \; y_1 \geq \mathfrak{a}_+. 
    \end{align*}
    The conclusion is reached by coming back into the $x$ variable, along with the fact that $(\sigma (y))^{-1} \vartheta' (y) = \theta' (x)$ in $L^2$.

    As for the last property, if $\theta'^2 -\cos^2 (\theta) = 0$ in $L^1$, then from classic ODE techniques $\cos \theta$ and $\theta'$ can never vanish unless $\theta$ is constant and $\theta$ is actually smooth. Thus, differentiating this expression, we get
    \begin{equation*}
        \theta'' + \cos \theta \sin \theta = 0.
    \end{equation*}
    Then \eqref{expr:SS_infinite:div a grad formulation} becomes
    \begin{equation*}
        s' \theta' = 0 \quad \text{ in } \mathcal{D}'.
    \end{equation*}
    As $\theta'$ is smooth and never vanishes, there holds $s' \equiv 0$ in $\mathcal{D}'$, which means that $s$ is constant and is therefore equal to $1$.
\end{proof}

\subsubsection{Strict monotonicity of the critical points}

\begin{lemma} \label{lem:crit_pt_monotony}
    If $s \in \mathcal{A}_a (\mathbb{R})$ is a symmetric notch, then any critical point $\theta$ of the energy $E_s$ in ${\mathscr W}_{\neq}$ is strictly monotone.
\end{lemma}

\begin{proof}
    The functions $\vartheta$ and $\sigma$ as defined in \eqref{eq:change_var_theta} satisfy \eqref{eq:ODE_vartheta}.
    This shows that $\vartheta'' \in L^\infty (\mathbb{R})$. We already know that $\theta \in L^\infty (\mathbb{R})$ from $\theta' \in L^2 (\mathbb{R})$ and $\lim_{\pm \infty} \abs{\theta} < \infty$, so that $\vartheta \in L^\infty (\mathbb{R})$. Therefore $\vartheta \in W^{2, \infty} (\mathbb{R})$, and in particular $\vartheta' \in \mathcal{C} (\mathbb{R})$.
    On the other hand, we also know from \eqref{expr:energy_eq_theta} that $\theta'^2 - \cos^2 \theta$ is non-negative in $\mathbb{R}$. We also have
    \begin{equation} \label{eq:vartheta_deriv}
        \theta' (x) = \frac{1}{s(x)} \vartheta' \bigg(\int_0^x\frac{\diff u}{s(u)}\bigg).
    \end{equation}
    Thus we get for all $y \in \mathbb{R}$
    \begin{equation} \label{ineq:energy_crit_pt}
        \vartheta'^2 (y) \geq \sigma (y)^2 \cos^2 \vartheta (y).
    \end{equation}
    From this, we know that $\vartheta'$ does not change sign. Indeed, if so, then, from the continuity of $\vartheta'$, we have $y_0 \in \mathbb{R}$ such that $\vartheta' (y_0) = 0$, which leads to $\cos \vartheta (y_0) = 0$ with \eqref{ineq:energy_crit_pt}. Solving the Cauchy problem of \eqref{eq:ODE_vartheta} with this initial condition leads to $\vartheta = cste$, which is in contradiction with the limits at $\pm \infty$ of $\theta$. Therefore, from \eqref{eq:vartheta_deriv}, $\theta'$ has the same sign all over $\mathbb{R}$, which leads to the conclusion. 
\end{proof}

\subsubsection{Strict coercivity of a Schrödinger operator}\label{sec:schro}
	
In \cite[Theorem 1]{book_carbou} and \cite[Theorem 1.1]{carbou2018stabilization}, existence of a (stationary) solution to Problem~\eqref{expr:SSTheta_infinite} is proved for admissible sections $s\in \mathscr{A}_a(\R)$. Furthermore, this solution is stable and asymptotically stable modulo rotations around the wire axis. The proof of this result rests upon the rewriting of System~\eqref{eq:LLG} in the mobile frame $(\boldsymbol{M}_0(x),\boldsymbol{M}_1(x),\boldsymbol{M}_2)$, with
\begin{equation*}
\boldsymbol{M}_0(x)=\begin{pmatrix}
\sin \theta(x)\\ \cos \theta(x)\\ 0
\end{pmatrix}, \quad
\boldsymbol{M}_1(x)=\begin{pmatrix}
-\cos \theta(x)\\ \sin \theta(x)\\ 0
\end{pmatrix}, \quad
\boldsymbol{M}_1=\begin{pmatrix}
0\\ 0\\1
\end{pmatrix},
\end{equation*}
where $\theta$ solves \eqref{expr:SSTheta_infinite}. Furthermore, by introducing $\boldsymbol{r}=[r_1,r_2]^\top$ such that
\begin{equation*}
\boldsymbol{m}(t,x)=\sqrt{1-|\boldsymbol{r}(t,x)|^2}\boldsymbol{M}_0(x)+r_1(t,x)\boldsymbol{M}_1(x)+r_2(t,x)\boldsymbol{M}_2,
\end{equation*}
one shows that Eq.~\eqref{eq:LLG} can be recast as
\begin{equation*}
\frac{\partial}{\partial t}\begin{pmatrix} r_1\\ r_2\end{pmatrix}=\begin{pmatrix} -L_1 & -L_2\\ L_1 & -L_2\end{pmatrix}\begin{pmatrix} r_1\\ r_2\end{pmatrix}+R(\boldsymbol{r}),
\end{equation*}
where $R$ is shown to be a remainder term in the stability analysis and 
\begin{equation*}
\left\{\begin{array}{l}
L_1(r)\coloneqq-\partial_{xx}r-\frac{s'}{s}\partial_xr +(\sin^2\theta-\cos ^2\theta)r, \\
L_2(r)\coloneqq-\partial_{xx}r-\frac{s'}{s}\partial_xr +(\sin^2\theta-\theta'^2)r.
\end{array}\right.
\end{equation*}
On $L^2(\R)$, let us introduce the weighted inner product $\langle \cdot,\cdot\rangle_s$ given by
\begin{equation*}
\langle u,v\rangle_s=\int_\R u(x)v(x)\, s(x)\diff x .
\end{equation*}
We recall some properties of $L_1$ and $L_2$ proved in \cite{carbou2018stabilization}.

%%%%%%
\begin{lemma} \label{lem:prop_L1_L2}
    If $\theta$ is such that $\theta' \in L^2$ and satisfies \eqref{expr:SSTheta_infinite} and $\theta (x) \in (- \frac{\pi}{2}, \frac{\pi}{2})$ for all $x \in \mathbb{R}$, then
    \begin{itemize}
        \item $L_2=\ell^*\ell$ where $\ell=\partial_x+\theta' \tan\theta$, and $L_2$ is non-negative with $\operatorname{ker}L_2=\R (\cos \theta)$.
        \item $L_1$ is a positive definite operator, and
        \begin{equation}\label{estim:stab:L1}
        \exists \alpha>0\quad \mid\quad \forall u\in H^2(\R), \quad  \langle L_1u,u\rangle_s\geq \alpha \Vert u\Vert_s^2, 
        \end{equation}
        where $\Vert u\Vert_s^2=\langle u,u\rangle_s$.
    \end{itemize} 
\end{lemma}

\begin{proof}
    One easily shows that for any $u$, $v$ in $H^2(\R)$, one has
    \begin{equation*}
    \langle L_2u,v\rangle_s=\langle \ell u,\ell v\rangle_s, \quad \text{with }\ell=\partial_x+\theta' \tan\theta .
    \end{equation*}
    This implies $L_2=\ell^*\ell$ and we infer that $L_2$ is non-negative. Moreover, it is easy to see that $L_2 \cos \theta = 0$, and it never vanishes on $\mathbb{R}$, which implies that $\operatorname{ker}L_2=\R (\cos \theta)$ as $L_2$ is a Schrödinger operator. It is also shown in \cite{carbou2018stabilization} that the essential spectrum of $L_2$ is $[1/2,+\infty)$.
    
    Regarding now $L_1$, observe that $L_1=L_2+(\theta'^2- \cos ^2\theta)\operatorname{Id}$.
    We also know that $\theta'^2-\cos^2\theta$ is non-negative in $\R$, but $\theta'^2-\cos^2\theta \not\equiv 0$ in $L^1$. %and positive in $[\gamma_-, \gamma_+]$ where $\gamma_- \coloneqq \inf \supp s$ and $\gamma_+ \coloneqq \sup \supp s$ (see \cite[Section 3.2.2]{carbou2018stabilization}).
    Thus we infer that $L_1$ is positive definite on $H^2(\R)$ and that \eqref{estim:stab:L1} holds.
\end{proof}

\begin{remark}
    The fact that $\operatorname{ker}L_2=\R (\cos \theta)$ corresponds to the invariance of the equation \eqref{eq:LLG} with respect to rotations around the nanowire axis.
\end{remark}

\begin{center}
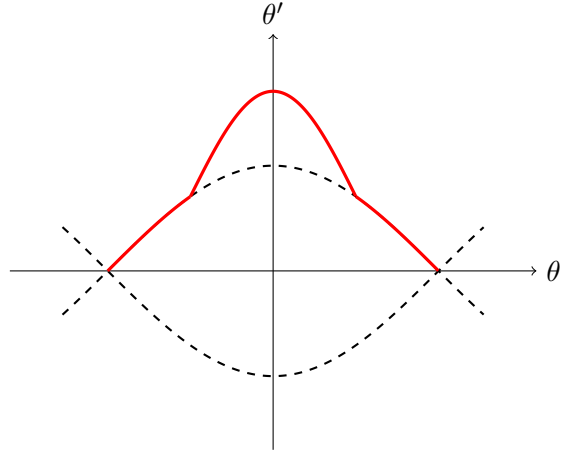
\begin{figure}[h!]
\centering
\begin{tikzpicture}[scale=1.4]
    % Définir les styles
    \tikzset{cosStyle/.style={black, thick}}
    \tikzset{highlightStyle/.style={red, very thick}}

    % Axe des abscisses et des ordonnées
    \draw[->] (-2.5, 0) -- (2.5, 0) node[right] {$\theta$};
    \draw[->] (0, -1.7) -- (0, 2.25) node[above] {$\theta'$};

    % Tracer cos(x) et -cos(x) sur [-2, 2]
    \draw[domain=-2:2, samples=200, cosStyle,dashed] plot (\x, {cos(deg(\x))}) node[above right] {};
    \draw[domain=-2:2, samples=200, cosStyle,dashed] plot (\x, {-cos(deg(\x))}) node[below right] {};

    % Fonction rouge épaisse
    % cos(x) sur [-pi/2, -pi/4]
    \draw[domain=-1.5708:-0.7854, samples=200, highlightStyle] plot (\x, {cos(deg(\x))});
    % sqrt(2)/2*cos(x) sur [-pi/4, pi/4]
    \draw[domain=-0.7854:0.7854, samples=200, highlightStyle] plot (\x, {sqrt(2)/2 + cos(2*deg(\x))});
    % cos(x) sur [pi/4, pi/2]
    \draw[domain=0.7854:1.5708, samples=200, highlightStyle] plot (\x, {cos(deg(\x))});
\end{tikzpicture}
\caption{Phase portrait of a notched nanowire: (dashed line) separatrix lines are modified (continuous line). $\mathcal{D}_\eta$.}	
\end{figure}
\end{center}

%%%%%%%%%%%%%%%%%%%%%%%%%%%%%%%%%%%%%%%%%%%%%%%%%%%%%%%%%%%%%%
%%%%%%%%%%%%%%%%%%%%%%%%%%%%%%%%%%%%%%%%%%%%%%%%%%%%%%%%%%%%%%
\subsection{Critical point and strict local minimizer}

%%%%%
\begin{lemma} \label{lem:crit_pt_strict_loc_min}
    Each critical point of $E_s$ is a strict local minimum: for every $\theta_0$ critical point of $E_s$ in ${\mathscr W}_{\neq}$, there exists $\varepsilon > 0$ and $\alpha_2 > 0$ such that, for every $\theta \in H^1 (\mathbb{R})$ such that $\norm{\theta}_{L^\infty} < \varepsilon$, then $E_s (\theta_0 + \theta) - E_s (\theta_0) \geq \alpha_2 \norm{\theta}_s^2$.
\end{lemma}
\begin{proof}
The fact that $E_s$ is three times differentiable is standard. Let us roughly compute its differentials. There holds
\begin{eqnarray*}
\diff E_s (\theta) . h &=& \int_\R(\theta' h's+\cos \theta\sin \theta hs)\, \diff x\\
&=& \int_\R h(-(s\theta')'+s\cos \theta\sin \theta)\, \diff x
\end{eqnarray*}
We thus infer that
\begin{eqnarray}\label{eq:d2_Es}
\diff^2 E_s (\theta) . (h_1,h_2) &=& \int_\R h_1\left(-(sh_2')'+sh_2\cos (2\theta)\right)\, \diff x=\langle L_1 h_1, h_2 \rangle_s
\end{eqnarray}
where the operator $L_1$ has been introduced in Section~\ref{sec:schro}, and finally,
\begin{eqnarray*}
\diff^3 E_s (\theta) . (h_1,h_2,h_3) &=& -2\int_\R h_1h_2h_3s\sin (2\theta)\, \diff x.
\end{eqnarray*}
 It follows that
    \begin{equation} \label{eq:est_d3_Hs}
        \abs{\diff^3 E_s (\theta) . (h_1, h_2, h_3)} \leq 2 \norm{h_1}_{L^2 (s(x) \diff x)} \norm{h_2}_{L^2 (s(x) \diff x)} \norm{h_3}_{L^\infty}
    \end{equation}
    Thus, there holds
    \begin{equation*}
        E_s (\theta_0 + \theta) - E_s (\theta_0) - \diff E_s (\theta_0) . \theta - \frac{1}{2} \diff^2 E_s (\theta_0) . (\theta, \theta) = \frac{1}{2} \int_0^1 (1-t)^2 \, \diff^3 E_s (\theta_0 + t \theta) (\theta, \theta, \theta) \diff s,
    \end{equation*}
    and, with \eqref{eq:est_d3_Hs} and the assumption $\norm{\theta}_{L^\infty} < \varepsilon$,
    \begin{align*}
        \abs{E_s (\theta_0 + \theta) - E_s (\theta_0) - E_s' (\theta_0) . \theta - \frac{1}{2} \diff^2 E_s (\theta_0) . (\theta, \theta)} &\leq \frac{1}{3} \norm{\theta}_{L^2 (s(x) \diff x)}^2 \norm{\theta}_{L^\infty} \\
            &\leq \frac{\varepsilon}{3} \norm{\theta}_{L^2 (s(x) \diff x)}^2 %\norm{\theta}_{H^1 (s(x) \diff x)}
    \end{align*}
    On the other hand, if $\theta_0$ is a critical point of $E_s$, then $E_s' (\theta_0) . \theta = 0$ and more precisely $\theta_0''+\frac{s'}{s}\theta_0'+\cos \theta_0 \sin \theta_0 = 0$. Then, from this property and Lemma \ref{lem:prop_L1_L2}, \eqref{eq:d2_Es} shows that
    \begin{equation*}
        E_s'' (\theta_0) . (\theta, \theta) > \alpha \norm{\theta}_{L^2 (s(x) \diff x)}^2.
    \end{equation*}
    By taking $\varepsilon = \alpha$, we get the conclusion with $\alpha_2 = \frac{\alpha}{6}$.
\end{proof}

\subsection{Transformation of the energy and convexity}

$E_s$ is not convex, because of $\cos^2 \theta$ in the potential energy, which is not convex itself. However, we have a way to make this function somehow convex. For this, we introduce the function
\begin{align*}
    L : \mathbb{R} \times (0, 1) &\rightarrow \mathbb{R} \\ 
    (y, z) &\mapsto \frac{y^2}{1 - z^2}.
\end{align*}
and the functional spaces, defined for any $x_0 \in \mathbb{R}$ by
\begin{gather} \label{eq:func_space_partial_new_energy_-}
    \mathcal{J}_{x_0}^- \coloneqq \{ \psi \in L^2 \cap H^1_\textnormal{loc} (-\infty, x_0) \, | \, 0 < \psi (x) < 1 \text{ for \textit{a.e.} } x \in (-\infty, x_0) \text{ and } L(\psi', \psi) \in L^1 (-\infty, x_0) \}, \\
    \mathcal{J}_{x_0}^+ \coloneqq \{ \psi \in L^2 \cap H^1_\textnormal{loc} (x_0, \infty) \, | \, 0 < \psi (x) < 1 \text{ for \textit{a.e.} } x \in (x_0, \infty) \text{ and } L(\psi', \psi) \in L^1 (x_0, \infty) \},
    \label{eq:func_space_partial_new_energy_+}
\end{gather}
from which we define the following partial functionals
\begin{gather} \label{eq:partial_new_energy_-}
    \mathcal{E}_{x_0}^- (\psi) = \frac{1}{2} \int_{- \infty}^{x_0} \Bigl( L(\psi' (x), \psi(x)) + \psi(x)^2 \Bigr) s(x) \diff x, \\
    \mathcal{E}_{x_0}^+ (\psi) = \frac{1}{2} \int_{x_0}^{\infty} \Bigl( L(\psi' (x), \psi(x)) + \psi(x)^2 \Bigr) s(x) \diff x,
    \label{eq:partial_new_energy_+}
\end{gather}
respectively defined on $\mathcal{J}_{x_0}^-$ and $\mathcal{J}_{x_0}^+$.
The reason why we introduce this functional is the fact that it is related to our initial energy $E_s (\theta)$, or more specifically to its partial integrals
\begin{gather}
    E_{s, x_0}^- \coloneqq \frac{1}{2} \int_{-\infty}^{x_0} \Bigl( \theta'(x)^2 + \cos^2\theta(x) \Bigr) s(x) \, \diff x, \\
    E_{s, x_0}^+ \coloneqq \frac{1}{2} \int_{x_0}^{\infty} \Bigl( \theta'(x)^2 + \cos^2\theta(x) \Bigr) s(x) \, \diff x.
\end{gather}

\begin{lemma} \label{lem:relation_energies}
    For any $\theta \in \mathscr{W}$ such that $- \frac{\pi}{2} < \theta (x) < 0$ for \textit{a.e.} $x \in (-\infty, x_0)$, there holds
    \begin{equation*}
        E_{s, x_0}^- (\theta) = \mathcal{E}_{x_0}^- (\cos \theta).
    \end{equation*}
    A similar property holds for $E_{s, x_0}^+$, with $0 < \theta (x) < \frac{\pi}{2}$ for \textit{a.e.} $x \in (x_0, \infty)$.
\end{lemma}

The proof is obvious and left to the reader. We also point out that, when $- \frac{\pi}{2} < \theta (x) < 0$, then $\theta (x) = - \arccos \cos \theta (x)$, whereas we have $\theta (x) = \arccos \cos \theta (x)$ when $0 < \theta (x) < \frac{\pi}{2}$.

We continue by some simple properties on $L$.

\begin{lemma}
    $L$ is a non-negative convex function on $\mathbb{R} \times (0,1)$.
\end{lemma}

\begin{proof}
    From the definition of $L$, it is indeed positive as soon as $z^2 < 1$. As for the convexity, we point out that
    \begin{equation*}
        \operatorname{Hess} L (y, z) =
        \begin{pmatrix}
            \frac{2}{1-z^2} & \frac{4 y z}{(1-z^2)^2} \\
            \frac{4 y z}{(1-z^2)^2} & \frac{2 y^2}{(1-z^2)^2} + \frac{8 y^2 z^2}{(1-z^2)^3}.
        \end{pmatrix}
    \end{equation*}
    In particular, it is easy to see that its trace is non-negative for all $y \in \mathbb{R}$ and $z \in (0, 1)$. As for the determinant, we get
    \begin{align*}
        \operatorname{det} \operatorname{Hess} L (y, z) &= \frac{2}{1-z^2} \biggl[ \frac{2 y^2}{(1-z^2)^2} + \frac{8 y^2 z^2}{(1-z^2)^3} \biggr] - \biggl( \frac{4 y z}{(1-z^2)^2} \biggr)^2 \\
            &= \frac{4 y^2}{(1-z^2)^3} \geq 0.
    \end{align*}
    This proves that $\operatorname{Hess} L (y, z)$ is a non-negative symmetric form for all $y \in \mathbb{R}$ and $z \in (0, 1)$, and thus that $L$ is convex on this space.
\end{proof}

Such a property on $L$ induce the convexity of the partial functionals $\mathcal{E}^\pm_{x_0}$ on $\mathcal{J}^\pm_{x_0}$.

\begin{lemma} \label{lem:convexity_part_energy}
    $\mathcal{J}_{x_0}^\pm$ are convex sets and $\mathcal{E}_{x_0}^\pm$ are convex functionals on these respective sets.
\end{lemma}

\begin{proof}
    We will focus on $\mathcal{J}_{x_0}^-$ and $\mathcal{E}_{x_0}^-$. The proof can be easily adapted for $\mathcal{J}_{x_0}^+$ and $\mathcal{E}_{x_0}^+$.
    
    Let $\psi_0, \psi_1 \in \mathcal{J}_{x_0}^-$ and $\lambda \in (0, 1)$, and define $\psi_\lambda \coloneqq (1-\lambda) \psi_0 + \lambda \psi_1 \in L^2 \cap H^1_\textnormal{loc} (-\infty, x_0)$. Of course, $0 < \psi_\lambda (x) < 1$ for \textit{a.e.} $x \in (- \infty, x_0)$. We know that, due to the convexity of $L$,
    \begin{equation*}
        L(\psi_\lambda' (x), \psi_\lambda (x)) \leq (1-\lambda) L(\psi_0' (x), \psi_0 (x)) + \lambda L(\psi_1' (x), \psi_1 (x)).
    \end{equation*}
    Since $L$ is non-negative on $\mathbb{R} \times (0, 1)$, this shows that $L(\psi_\lambda', \psi_\lambda) \in L^1 (-\infty, x_0)$, which proves that $\psi_\lambda \in \mathcal{J}_{x_0}^\pm$, and that
    \begin{equation*}
        \int_{- \infty}^{x_0} L(\psi_\lambda' (x), \psi_\lambda (x)) s(x) \diff x \leq (1-\lambda) \int_{- \infty}^{x_0} L(\psi_0' (x), \psi_0 (x)) s(x) \diff x + \lambda \int_{- \infty}^{x_0} L(\psi_1' (x), \psi_1 (x)) s(x) \diff x.
    \end{equation*}
    On the other hand, we also have
    \begin{equation*}
        \int_{- \infty}^{x_0} (\psi_\lambda (x))^2 s(x) \diff x \leq (1 - \lambda) \int_{- \infty}^{x_0} (\psi_0 (x))^2 s(x) \diff x + \lambda \int_{- \infty}^{x_0} (\psi_1 (x))^2 s(x) \diff x.
    \end{equation*}
    Thus,
    \begin{equation*}
        \mathcal{E}_{x_0}^- (\psi_\lambda) \leq (1-\lambda) \mathcal{E}_{x_0}^- (\psi_0) + \lambda \mathcal{E}_{x_0}^- (\psi_1). \qedhere
    \end{equation*}
\end{proof}

\subsection{Construction of a particular path}

From such properties, we can construct some interesting paths connecting a critical point to a translated basic domain wall without notch, whose angle is given by $\theta_* (x) = \arctan (\sinh(x))$.

In particular, we have
\begin{equation*}
    \cos \theta_\ast (x) = \frac{1}{\cosh (x)}\qquad\text{and}\qquad\sin\theta_\ast(x)=\tanh(x).
\end{equation*}

\begin{lemma} \label{lem:part_path_1}
    Let $\theta_0 \in W$ be a critical point of $E_s$ and $x_0 \in \mathbb{R}$ such that $\theta_0 (x_0) = 0$. For any $\lambda \in [0, 1]$ and $x \in \mathbb{R}$, we define
    \begin{equation*}
        \mathcal{P}_\lambda (\theta_0) (x) = \operatorname{sgn} (x - x_0) \, \arccos \Bigl( (1 - \lambda) \cos \theta_0 (x) + \lambda \cos \theta_* (x - x_0) \Bigr).
    \end{equation*}
    Then $\lambda\mapsto \mathcal{P}_\lambda (\theta_0)$ is a path with values in $W$, continuous with respect to the $H^1$ norm, connecting $\theta_0$ and $\theta_* (\cdot - x_0)$, and for any $\lambda \in (0, 1)$,
    \begin{equation} \label{eq:convexity_path_1}
        E_s (\mathcal{P}_\lambda (\theta_0)) \leq (1 - \lambda) E_s (\theta_0) + \lambda E_s (\theta_* (\cdot - x_0)).
    \end{equation}
\end{lemma}

\begin{proof}
    First, from Lemma \ref{lem:crit_pt_monotony}, we know that $- \frac{\pi}{2} < \theta_0 (x) < 0$ for all $x \in (- \infty, x_0)$ and $0 < \theta_0 (x) < \frac{\pi}{2}$ for all $x \in (x_0, \infty)$. The same property holds for $\theta_* (x - x_0)$ and for $\mathcal{P}_\lambda (\theta_0) (x)$. Then, using Lemma \ref{lem:relation_energies}, it is easy to see that $\cos \theta_0 \in \mathcal{J}_{x_0}^\pm$, and so is $\cos \theta_*$, therefore so is $\cos \mathcal{P}_\lambda (\theta_0) = (1 - \lambda) \cos \theta_0 + \lambda \cos \theta_*$ by the convexity of these spaces due to Lemma \ref{lem:convexity_part_energy}.
    Concerning the regularity of the map $\lambda\mapsto \mathcal{P}_\lambda(\theta_0)\in W$, the only point that may create a singularity is the multiplicative signum function. 
    This is not the case here as a consequence of Lemma~\ref{lem:lipschitz lemma} in Appendix~\ref{sec:append}. To apply this lemma, it is easy to check that for all $\lambda \in [0, 1]$, the function
    \begin{equation*}
        x\longmapsto\arccos \Bigl( (1 - \lambda) \cos \theta_0 (x) + \lambda \cos \theta_* (x - x_0) \Bigr),
    \end{equation*}
    vanishes when $x=x_0$. 
    As a consequence of Lemma \ref{lem:lipschitz lemma} but also Lemma \ref{lem:composition ponctual estimates} and its corollary \ref{coro:stability for L p} in Appendix~\ref{sec:append}, we can write  $\lambda\mapsto \mathcal{P}_\lambda(\theta_0)\in W$ as a composition of continuous functions (for $W$ being endowed with the $\dot{H}^1$ norm).
    
    For the estimate, decomposing $E_s (\mathcal{P}_\lambda (\theta_0))$ leads to
    \begin{align*}
        E_s (\mathcal{P}_\lambda (\theta_0)) &= E_{s, x_0}^- (\mathcal{P}_\lambda (\theta_0)) + E_{s, x_0}^+ (\mathcal{P}_\lambda (\theta_0)) \\
            &= \mathcal{E}_{x_0}^- (\cos \mathcal{P}_\lambda (\theta_0)) + \mathcal{E}_{x_0}^+ (\cos \mathcal{P}_\lambda (\theta_0)) \\
            &\leq \mathcal{E}_{x_0}^- ((1 - \lambda) \cos \theta_0 + \lambda \cos \theta_* (\cdot - x_0)) + \mathcal{E}_{x_0}^+ ((1 - \lambda) \cos \theta_0 + \lambda \cos \theta_* (\cdot - x_0)).
    \end{align*}
    By using Lemma \ref{lem:convexity_part_energy} and Lemma \ref{lem:relation_energies} once again, we get
    \begin{align*}
        E_s (\mathcal{P}_\lambda (\theta_0)) &\leq (1 - \lambda) \mathcal{E}_{x_0}^- (\cos \theta_0) + \lambda \mathcal{E}_{x_0}^- (\cos \theta_* (\cdot - x_0)) + (1 - \lambda) \mathcal{E}_{x_0}^+ (\cos \theta_0) + \lambda \mathcal{E}_{x_0}^+ (\cos \theta_* (\cdot - x_0)) \\
            &\leq (1 - \lambda) E_{s, x_0}^- (\theta_0) + \lambda E_{s, x_0}^- (\theta_* (\cdot - x_0)) + (1 - \lambda) E_{s, x_0}^+ (\theta_0) + \lambda E_{s, x_0}^+ (\theta_* (\cdot - x_0)) \\
            &\leq (1 - \lambda) E_s (\theta_0) + \lambda E_s (\theta_* (\cdot - x_0)). \qedhere
    \end{align*}
\end{proof}

\begin{lemma} \label{lem:energy_translatedDW}
    The function $\gamma \mapsto \theta_* (\cdot - \gamma)$, which takes values in $W$, is continuous with respect to the $H^1$ topology. Moreover, for any $s \in \mathscr{U}_a (\mathbb{R})$ such that $s(\cdot) \neq 1$ and for any $\gamma_* \in \mathbb{R}$,
    \begin{equation*}
        \max_{\gamma \in [- \gamma_*, \gamma_*]} E_s (\theta_* (\cdot - \gamma)) < E_1 (\theta_*).
    \end{equation*}
    Last,
    \begin{equation*}
        \lim_{\gamma \rightarrow \pm \infty} E_s (\theta_* (\cdot - \gamma)) = E_1 (\theta_*)
    \end{equation*}
\end{lemma}

\begin{proof}
    Since $s \leq 1$ and $\theta_*$ is strictly increasing (in particular it is not constant on $\supp (s-1)$, which is of positive measure), it is easy to show that $E_s (\theta_* (\cdot - \gamma)) < E_1 (\theta_* (\cdot - \gamma)) = E_1 (\theta_*)$.
    Moreover,
    \begin{align*}
        E_s (\theta_* (\cdot - \gamma)) &= \frac{1}{2} \int \Bigl( \abs{\partial_x \theta_* (x)}^2 + \cos^2 (\theta_* (x)) \Bigr) \, s (x + \gamma) \diff x \\
            &\underset{\abs{\gamma} \rightarrow \infty}{\longrightarrow} \frac{1}{2} \int \Bigl( \abs{\partial_x \theta_* (x)}^2 + \cos^2 (\theta_* (x)) \Bigr) \diff x,
    \end{align*}
    as it is easy to prove that, since $s - 1$ is compactly supported, there holds $s (\cdot - \gamma) - 1 \rightharpoonup 0$ as $\abs{\gamma} \rightarrow \infty$ for the Radon measure topology $\mathcal{M} (\mathbb{R})$ but also for the local uniform convergence.

    Also, there also holds $\partial_x \theta_* \in L^2 (\mathbb{R})$ and $\cos \theta_* \in L^2 (\mathbb{R})$, thus $\gamma \mapsto \partial_x \theta_* (\cdot - \gamma)$ and $\gamma \mapsto \cos \theta_* (\cdot - \gamma)$ are continuous with respect to the $L^2$ norm, which gives the continuity of $\gamma \mapsto E_s (\theta_* (\cdot - \gamma))$.
    We also point out that $\theta_*$ converges exponentially to its limits at $\pm \infty$, so that for any $\gamma \in \mathbb{R}$, $\theta_* (\cdot - \gamma) - \theta_* \in L^2$, and the continuity of $\gamma \mapsto \theta_* (\cdot - \gamma)$ with respect to the $L^2$ norm can be easily proven.
\end{proof}

We are now in position to bound from above the energy of the critical points by comparison with the notchless case:

\begin{lemma} \label{lem:ineq_energy_crit_pt}
    Any critical point $\theta_0$ of $E_s$ satisfies $E_s (\theta_0) \leq E_s (\theta_* (\cdot - x_0)) < E_1 (\theta_*)$, where $x_0 \in \mathbb{R}$ is such that $\theta_0 (x_0) = 0$.
\end{lemma}

\begin{proof}
    The second inequality is given by Lemma \ref{lem:energy_translatedDW}.
    On the other hand, we know that $\mathcal{P}_\lambda (\theta_0)$ defined in Lemma \ref{lem:part_path_1} is a path continuous with respect to the $H^1$ topology. Therefore, by Lemma \ref{lem:crit_pt_strict_loc_min}, there exists $\lambda_0 > 0$ such that, for every $\lambda \in (0, \lambda_0)$, $E_s (\mathcal{P}_\lambda (\theta_0)) \geq E_s (\theta_0)$. Putting these inequalities in \eqref{eq:convexity_path_1}, we find
    \begin{equation*}
        E_s (\theta_0) \leq (1 - \lambda) E_s (\theta_0) + \lambda E_s (\theta_* (\cdot - x_0)),
    \end{equation*}
    for any $\lambda \in (0, \lambda_0)$. Thus,
    \begin{equation*}
        0 \leq \lambda ( E_s (\theta_* (\cdot - x_0)) - E_s (\theta_0) ),
    \end{equation*}
    which leads to the conclusion.
\end{proof}

\begin{corollary} \label{cor:max_energy_path_1}
    For any critical point $\theta_0$ of the energy $E_s$, the maximal energy of the path $\mathcal{P}_\lambda (\theta_0)$ is $\max_{\lambda \in [0, 1]} E_s (\mathcal{P}_\lambda (\theta_0)) = E_s (\theta_* (\cdot - x_0))$, where $x_0$ is such that $\theta_0 (x_0) = 0$.
\end{corollary}

\begin{proof}
    Combining \eqref{eq:convexity_path_1} and Lemma \ref{lem:ineq_energy_crit_pt}, we get that $E_s (\mathcal{P}_\lambda (\theta_0)) \leq E_s (\theta_* (\cdot - x_0))$ for all $\lambda \in [0, 1]$. The conclusion comes from the fact that $\mathcal{P}_1 (\theta_0) = \theta_* (\cdot - x_0)$.
\end{proof}

The set $\mathfrak{P}$ of continuous path on $[0, 1]$ connecting the critical point $\theta_0$ to the minimizer $\theta_s$ of the energy $E_s$ is defined as an affine subspace included in $\mathcal{C} ([0, 1], W)$, where $W$, as defined in \eqref{eq:def_W}, is an affine space whose associated vector space is $H^1$. Thus, $\mathfrak{P}$ is naturally endowed with a distance inherited from the $H^1$ norm.

The structure of our proof can be easily guessed : we would like to use the so-called {\it mountain pass theorem} between $\theta_0$ and $\theta_s$ to find a new critical point of the energy $E_s$. Such a function should be a saddle point, but it is actually a strict local minimizer due to Lemma \ref{lem:crit_pt_strict_loc_min}, giving a contradiction with the existence of another critical point than $\theta_s$.

From such a tactic, it is classic to show that the functional satisfies the Palais-Smale property. However, such a property is false for $E_s$ generally speaking : one can think of $\theta_* ( . - n)$, which is a notchless domain wall for the nanowire without notch translated to infinity, and whose energy converge to $E_1 (\theta_*)$. Nonetheless, this is probably the "minimal" counter-example which can be found, and such a thing can not happen in our problem. To be more precise, first, from all the previous results, we can construct a path, whose maximal energy is \emph{strictly lower} than the notchless energy of the notchless domain wall.

\begin{lemma} \label{lem:path_small_energy}
    Let $\theta_s$ a minimizer of $E_s$ in $W$ and $\theta_0$ be any other critical point in ${\mathscr W}_0$. Define, for any $\lambda \in [0, 1]$,
    \begin{equation*}
        c_0 (\lambda) \coloneqq
        \begin{cases}
            \mathcal{P}_{3 \lambda} (\theta_0) & \text{ if } \lambda \in [0, \frac{1}{3}], \\
            \theta_* (\cdot - (2 - 3 \lambda) x_0) & \text{ if } \lambda \in (\frac{1}{3}, \frac{2}{3}), \\
            \mathcal{P}_{3 - 3 \lambda} (\theta_s) & \text{ if } \lambda \in [\frac{2}{3}, 1].
        \end{cases}
    \end{equation*}
    Then $c_0 (\lambda)$ is a path in $\mathcal{C} ([0, 1], W)$, connecting $\theta_0$ and $\theta_s$. Moreover, it satisfies
    \begin{equation*}
        \max_{\lambda \in [0, 1]} E_s (c_0 (\lambda)) = E_s (\theta_* (\cdot - x_0)) < E_1 (\theta_*).
    \end{equation*}
\end{lemma}

\begin{proof}
    This path is a continuous concatenation of three continuous paths  with respect to the $H^1$ norm, whose maximal energies are known to be respectively $E_s (\theta_* (\cdot - x_0))$ (due to Corollary \ref{cor:max_energy_path_1}),  $E_s (\theta_* (\cdot - x_0))$ (due to Lemma \ref{lem:energy_translatedDW}) and $E_s (\theta_*)$ (using again Corollary \ref{cor:max_energy_path_1} with $\theta_s$, which is a particular critical point). The conclusion follows from Lemma \ref{lem:energy_translatedDW} once again.
\end{proof}

\subsection{Palais-Smale property}

Due to this particular path, the problem is much more reduced : the Palais-Smale condition has to be proved only for sequences of functions whose energies are \emph{strictly smaller} than $E_1 (\theta_*)$. Actually, we can prove that, up to a sub-sequence, we have a \emph{strong} convergence with respect to the $H^1$ norm. But before that, we need two properties. The first one gives a weakly/strongly converging sub-sequence for any sequence whose energy is uniformly bounded.

\begin{lemma} \label{lem:weak_strong_conv}
    Let $\theta_n \in W$ such that $\sup_n E_s (\theta_n) < \infty$. Then, for any $x_n \in \mathbb{R}$, there exists $\overline{\theta} \in \mathscr{W}$ such that, up to an omitted extraction,
    \begin{itemize}
        \item $\theta_n (. + x_n) \longrightarrow \overline{\theta}$ in $L^2 (I)$ and $\mathcal{C} (I)$ for any bounded interval $I$,
        \item $\theta_n' (. + x_n) \rightharpoonup \overline{\theta}'$ in $L^2$,
        \item $\cos \theta_n (. + x_n) \rightarrow \cos \overline{\theta}$ in $\mathcal{C} (I)$ for any bounded interval $I$,
        \item $\cos \theta_n (. + x_n) \rightharpoonup \cos \overline{\theta}$ in $L^2$,
        \item $\sin (2 \theta_n (. + x_n)) \rightharpoonup \sin (2 \overline{\theta})$ in $L^2$,
        \item if $\abs{x_n} \rightarrow \infty$, then $\overline{\theta}$ can be chosen so that $E_1 (\overline{\theta}) \leq \liminf_{n\to+\infty}E_s (\theta_n)$.
    \end{itemize}
\end{lemma}

\begin{proof}
    The proof of all but the last point follows from the fact that $\theta_n' (. + x_n)$, $\cos \theta_n (. + x_n)$ and $\sin (2 (. + x_n))$ are bounded in $L^2$ and due to the compact Sobolev embeddings $H^1 (I) \hookrightarrow \mathcal{C} (I)$ for any bounded interval $I$.

    As for the last point, we can show in the case $\abs{x_n} \rightarrow \infty$ that $\sqrt{s (. + x_n)} \theta_n' (. + x_n) \rightharpoonup \overline{\theta}'$ in $L^2$. Indeed, $\sqrt{s (. + x_n)} \theta_n' (. + x_n)$ is uniformly bounded in $L^2$ and, since $s \equiv 1$ outside $[-a, a]$, converges weakly to $\overline{\theta}'$ in $L^2 (I)$ for every bounded interval $I$. The same thing can be done for $\sqrt{s (. + x_n)} \cos \theta_n (. + x_n)$ and thus
    \begin{align*}
        E_1 (\overline{\theta}) &= \norm{\overline{\theta}'}_{L^2}^2 + \frac{1}{2} \norm{\cos \overline{\theta}}_{L^2}^2 \\
            &\leq \liminf_{n\to+\infty}\norm{\sqrt{s (. + x_n)} \theta_n' (. + x_n)}_{L^2}^2 + \frac{1}{2} \liminf_{n\to+\infty}\norm{\sqrt{s (. + x_n)} \cos \theta_n (. + x_n)}_{L^2}^2 \\
            &\leq \liminf_{n\to+\infty}E_s (\theta_n). \qedhere
    \end{align*}
\end{proof}

Then, we need a property showing that, roughly speaking, no big bump can go to infinity if the energy is too low and if the sequence is close to be a critical point of $E_s$.

\begin{lemma} \label{lem:no_big_bump}
    Let $\theta_n \in W$ and $|x_n| \rightarrow +\infty$ such that $- \frac{3 \pi}{8} < \theta_n (x_n) < \frac{3 \pi}{8}$ and $(s \theta_n')' + \frac{s}{2} \sin (2 \theta_n) \rightarrow 0$ in $H^{-1}$. Then $\liminf_{n\to+\infty}E_s (\theta_n) \geq E_1 (\theta_*)$.
\end{lemma}

\begin{proof}
    We can assume that the energy of at least a sub-sequence of $\theta_n$ is bounded, otherwise the property is obvious. If so, take a sub-sequence (still denoted $\theta_n$) whose energy converge to $\liminf_{n\to+\infty}E_s (\theta_n)$. Then, we can use Lemma \ref{lem:weak_strong_conv} with $\theta_n (. + x_n)$, from which we deduce for the provided limit $\overline{\theta}$ of a sub-sequence that :
    \begin{itemize}
        \item $- \frac{3 \pi}{8} \leq \overline{\theta} (0) \leq \frac{3 \pi}{8}$,
        \item $E_1 (\overline{\theta}) \leq \liminf_{n\to+\infty}E_s (\theta_n)$.
    \end{itemize}
    Furthermore, for $\psi \in \mathcal{C}^\infty_c (\mathbb{R})$, $\psi (. - x_n)$ is uniformly bounded in $H^1$, so that
    \begin{equation*}
        \langle (s \theta_n')' + \frac{s}{2} \sin (2 \theta_n), \psi (. - x_n) \rangle \underset{n \rightarrow \infty}{\longrightarrow} 0.
    \end{equation*}
    On the other hand,
    \begin{align*}
        \langle (s \theta_n')' + \frac{s}{2} \sin (2 \theta_n), \psi (. - x_n) \rangle &= - \int s(x) \theta_n' (x) \psi' (x - x_n) \diff x + \frac{1}{2} \int s(x) \sin (2 \theta_n (x)) \psi (x - x_n) \diff x \\
            & \begin{multlined} =
                - \int s(x + x_n) \theta_n' (x + x_n) \psi' (x) \diff x \\ + \frac{1}{2} \int s(x + x_n) \sin (2 \theta_n (x + x_n)) \psi (x) \diff x
                \end{multlined} \\
        \intertext{and since $s(. + x_n) \equiv 1$ on $\supp \psi$ for $n$ large enough,}
            &= - \int \theta_n' (x + x_n) \psi' (x) \diff x + \frac{1}{2} \int \sin (2 \theta_n (x + x_n)) \psi (x) \diff x \\
            &\underset{n \rightarrow \infty}{\longrightarrow} - \int \overline{\theta}' (x) \psi' (x) \diff x + \frac{1}{2} \int \sin (2 \overline{\theta} (x)) \psi (x) \diff x.
    \end{align*}
    This proves that $\overline{\theta}'' + \frac{1}{2} \sin (2 \overline{\theta}) = 0$ in $\mathcal{D}' (\mathbb{R})$, and thus in $L^2$ and therefore $\overline{\theta} \in \mathcal{C}^\infty (\mathbb{R})$. In particular, $\overline{\theta}$ is a critical point of $E_1$, which is not constant since $- \frac{3 \pi}{8} \leq \overline{\theta} (0) \leq \frac{3 \pi}{8}$. Thus, up to a translation, $\overline{\theta} = \theta_*$ (\cite{Cote_Ignat__stab_DW_LLG_DM})
    and $\liminf_{n\to+\infty}E_s (\theta_n) \geq E_1 (\overline{\theta}) = E_1 (\theta_*)$.
\end{proof}

Now, we have all the materials to prove a Palais-Smale property as announced previously.

\begin{lemma}[Weak-strong Palais-Smale property] \label{lem:palais-smale}
    Let $\theta_n \in W$ such that
    \begin{itemize}
        \item $- \frac{5 \pi}{8} \leq \theta_n (x) \leq \frac{5 \pi}{8}$ for all $x \in \mathbb{R}$ and $n \in \mathbb{N}$,
        \item $\sup_{n} E_s (\theta_n) < E_1 (\theta_*)$,
        \item $(s \theta_n')' + \frac{s}{2} \sin (2 \theta_n) \underset{n \rightarrow \infty}{\longrightarrow} 0$ in $H^{-1} (\mathbb{R})$.
    \end{itemize}
    Then, up to a sub-sequence, there exists a critical point $\overline{\theta} \in W$ of $E_s$ such that $\norm{\theta_n - \overline{\theta}}_{H^1} \rightarrow 0$.
\end{lemma}

\begin{proof}
    Let $\overline{\theta}$ given by Lemma \ref{lem:weak_strong_conv} applied to $\theta_n$. Since we also have
    \begin{equation*}
        s \theta_n' \rightharpoonup s \overline{\theta}' \quad \text{and} \quad s \sin (2 \theta_n) \rightharpoonup s \sin (2 \overline{\theta}) \quad \text{in } L^2,
    \end{equation*}
    then we can easily shows that $(s \overline{\theta}')' + \frac{1}{2} s \sin (2 \overline{\theta}) = 0$, which shows that $\overline{\theta}$ is a critical point of $E_s$.
    On the other hand, from Lemma \ref{lem:no_big_bump}, there exists $L > 0$ such that
    \begin{itemize}
        \item for all $x \geq L$ and all $n \in \N$, $\frac{3 \pi}{8} \leq \theta_n (x) \leq \frac{5 \pi}{8}$,
        \item for all $x \leq - L$ and all $n \in \N$, $- \frac{5 \pi}{8} \leq \theta_n (x) \leq - \frac{3 \pi}{8}$.
    \end{itemize}
    Thus, such properties are still valid for $\overline{\theta}$. Since $\cos \overline{\theta} \in H^1 \subset \mathcal{C}_0 (\mathbb{R})$, we thus get $\lim_{\pm \infty} \overline{\theta} = \pm \frac{\pi}{2}$, so that $\overline{\theta} \in {\mathscr W}_0$, and thus $\overline{\theta} \in W$ by Lemma \ref{lem:crit_pt_monotony}. This proves that, for all $n \in \N$, $\theta_n - \overline{\theta} \in H^1$.
    From the strong convergence of $\theta_n$ in $L^2$ on bounded intervals, we also get
    \begin{equation*}
        B_n \coloneqq \int_{-L}^L (\theta_n - \overline{\theta})^2 \diff x \underset{n \rightarrow \infty}{\longrightarrow} 0.
    \end{equation*}
    Then, for any $y, z \in (\frac{3 \pi}{8}, \frac{5 \pi}{8})$, we know that
    \begin{equation*}
        \abs{\cos (y) - \cos (z)} \geq \frac{\sqrt{3}}{2} \abs{y - z}.
    \end{equation*}
    From the previous property, we get
    \begin{align*}
        \int_L^\infty (\theta_n - \overline{\theta})^2 \diff x &\leq \frac{2}{\sqrt{3}} \int_L^\infty \Bigl( \cos (\theta_n ) - \cos (\overline{\theta} ) \Bigr)^2 \diff x \\
            & \leq \frac{2}{\sqrt{3}} \norm{\cos \theta_n - \cos \overline{\theta}}_{L^2}^2 \\
            & \leq \frac{4}{\sqrt{3}} \Bigl( \norm{\cos \theta_n}_{L^2}^2 + \norm{\cos \overline{\theta}}_{L^2}^2 \Bigr) \\
            &\leq \frac{4}{\sqrt{3} s_0} (E_s (\theta_n) + E_s (\overline{\theta})) \\
            &\leq \frac{8}{\sqrt{3} s_0} E_1 (\theta_*).
    \end{align*}
    A similar property holds for the interval $(- \infty, -L)$.
    Obviously, there also holds
    \begin{equation*}
        \norm{(\theta_n - \overline{\theta})'}_{L^2} \leq \norm{\theta_n'}_{L^2} + \big\| \overline{\theta}'\big\|_{L^2} \leq \frac{1}{\sqrt{s_0}} \Bigl( \sqrt{E_s (\theta_n)} + \sqrt{E_s (\overline{\theta})} \Bigr) \leq \frac{2}{\sqrt{s_0}} \sqrt{E_1 (\theta_*)}.
    \end{equation*}
    Therefore, $\norm{\theta_n - \overline{\theta}}_{H^1}$ is bounded uniformly in $n$.
    Hence,
    \begin{equation*}
        \abs{\langle (s \theta_n')' + \frac{s}{2} \sin (2 \theta_n), \theta_n - \overline{\theta} \rangle} \leq \norm{(s \theta_n')' + \frac{s}{2} \sin (2 \theta_n)}_{H^{-1}} \norm{\theta_n - \overline{\theta}}_{H^1} \underset{n \rightarrow \infty}{\longrightarrow} 0.
    \end{equation*}
    However, we also have $(s \overline{\theta}')' + \frac{s}{2} \sin (2 \overline{\theta}) = 0$, which yields
    \begin{align*}
        A_n \coloneqq \langle (s \theta_n')' + \frac{s}{2} \sin (2 \theta_n), \theta_n - \overline{\theta} \rangle &= \langle (s (\theta_n' - \overline{\theta}'))' + \frac{s}{2} (\sin (2 \theta_n) - \sin (2 \overline{\theta})), \theta_n - \overline{\theta} \rangle \\
            &= \begin{multlined}[t]
                - \int s (\theta_n' - \overline{\theta}' )^2 \diff x \\ + \frac{1}{2} \int s (\sin (2 \theta_n) - \sin (2 \overline{\theta} )) (\theta_n  - \overline{\theta}) \diff x
                \end{multlined}
    \end{align*}
    We decompose the last integral into three intervals : $(-\infty, -L)$, $(-L, L)$ and $(L, \infty)$. From the property of the function $\sin (2 z)$, it is easy to estimate the second integral:
    \begin{equation*}
        \abs{\int_{-L}^L s \, (\sin (2 \theta_n) - \sin (2 \overline{\theta})) (\theta_n - \overline{\theta}) \diff x} \leq C \, B_n.
    \end{equation*}
    On the other hand, for the third integral, we know that, as soon as $y_1, y_2 \in (\frac{3 \pi}{8}, \frac{5 \pi}{8})$, there exists $y_3 \in (\frac{3 \pi}{8}, \frac{5 \pi}{8})$ such that
    \begin{equation*}
        \sin (2 y_1) - \sin (2 y_2) = 2 (y_1 - y_2) \cos (2 y_3),
    \end{equation*}
    and thus
    \begin{equation*}
        (\sin (2 y_1) - \sin (2 y_2)) (y_1 - y_2) = 2 (y_1 - y_2)^2 \cos (2 y_3) \leq - \sqrt{2} (y_1 - y_2)^2.
    \end{equation*}
    Therefore, assuming also $L \geq a$ so that $s \equiv 1$ on $(L, \infty)$, we obtain
    \begin{equation*}
        \int_L^\infty s \, (\sin (2 \theta_n) - \sin (2 \overline{\theta})) (\theta_n - \overline{\theta}) \diff x \leq - \sqrt{2} \int_L^\infty (\theta_n - \overline{\theta})^2 \diff x \leq 0.
    \end{equation*}
    A similar estimate holds for $(-\infty, -L)$. Gathering these estimates, we get
    \begin{align*}
        \int s (\theta_n' - \overline{\theta}')^2 \diff x + \frac{1}{\sqrt{2}} \int (\theta_n - \overline{\theta})^2 \diff x \leq A_n + C B_n + \frac{1}{\sqrt{2}} B_n \underset{n \rightarrow \infty}{\longrightarrow} 0.
    \end{align*}
    This proves that $\theta_n \rightarrow \overline{\theta}$ strongly in $H^1$.
\end{proof}

\subsection{A contradiction by the mountain pass theorem}

With such a Palais-Smale property, we are able to prove the uniqueness of the critical point of $E_s$ by contradiction through the method of the mountain pass theorem.

\begin{theorem}\label{theo:criticalPt}
    The critical point of $E_s$ in ${\mathscr W}_0$ is unique.
\end{theorem}

\begin{proof}
    We have already constructed $\theta_s$ in Proposition \ref{prop:strict monotony}: it is a minimizer of $E_s$ in ${\mathscr W}_0$ and therefore a critical point.
    By contradiction, let $\theta_0$ be another critical point.
    
    \medskip
    
    \underline{Step 1} : Setting up the mountain pass problem.
    
    Let $\mathfrak{P}$ be the set of continuous path (with respect to the $H^1$ norm) on $[0, 1]$ connecting $\theta_0$ to $\theta_s$. We want to solve the problem
    \begin{equation} \label{def:tilde_m}
        \tilde{\mathfrak{m}} \coloneqq \min_{c \in \mathfrak{P}} \max_{\lambda \in [0, 1]} E_s (c (\lambda)).
    \end{equation}
    We know that $E_s (\theta_0) < E_1 (\theta_*)$ by Lemma \ref{lem:ineq_energy_crit_pt}, and that there exists a path $c_0$ (constructed in Lemma \ref{lem:path_small_energy}) whose maximal energy is strictly smaller than $E_1 (\theta_*)$. Thus, $\tilde{\mathfrak{m}} < E_1 (\theta_*)$.

    On the other hand, let $c \in \mathfrak{P}$ such that $\max_{\lambda} E_s (c (\lambda)) < E_1 (\theta_*)$. We know that, for any $\lambda_1, \lambda_2 \in [0, 1]$,
    \begin{align*}
        \norm{c (\lambda_1) - c (\lambda_2)}_{L^\infty} &\leq C \norm{c (\lambda_1) - c (\lambda_2)}_{L^2}^\frac{1}{2} \norm{c (\lambda_1)' - c (\lambda_2)'}_{L^2}^\frac{1}{2} \\
            &\leq C \norm{c (\lambda_1) - c (\lambda_2)}_{L^2}^\frac{1}{2} \Bigl( \norm{c (\lambda_1)'}_{L^2} + \norm{c (\lambda_1)'}_{L^2}\Bigr)^\frac{1}{2} \\
            &\leq C \sqrt{\frac{2 E_1 (\theta_*)}{s_0}} \norm{c (\lambda_1) - c (\lambda_2)}_{L^2}^\frac{1}{2}.
    \end{align*}
    Take some $\delta > 0$ (independent on the path $c$ chosen) such that
    \begin{equation} \label{eq:delta_cond_1}
        \delta \le \frac{1}{2} \norm{\theta_0 - \theta_s}_{L^2}
    \end{equation}
    and
    \begin{equation} \label{eq:delta_cond_2}
        C \sqrt{\frac{2 E_1 (\theta_*)}{s_0}} \delta^\frac{1}{2} \le \min(\varepsilon_0, \varepsilon_s),
    \end{equation}
    where $\varepsilon_0$ (resp. $\varepsilon_s$) is given by applying Lemma \ref{lem:crit_pt_strict_loc_min} to $\theta_0$ (resp. $\theta_s$). We know that there exists $\lambda_1 \in [0, 1]$ such that $\norm{c (\lambda_1) - \theta_0}_{L^2} = \delta$ due to \eqref{eq:delta_cond_1} and the fact that $c (0) = \theta_0$. Then, \eqref{eq:delta_cond_2} along with Lemma \ref{lem:crit_pt_strict_loc_min} ensures that
    \begin{equation*}
        E_s (c (\lambda_1)) \ge E_s (\theta_0) + \alpha_2 \delta^2.
    \end{equation*}
    Similar arguments lead to
    \begin{equation*}
        E_s (c (\lambda_2)) \ge E_s (\theta_s) + \alpha_2 \delta^2.
    \end{equation*}
    This shows that $\max_{\lambda} E_s (c (\lambda)) > \alpha_2 \delta^2 + \max (E_s (\theta_0), E_s (\theta_s))$, and therefore
    \begin{equation} \label{eq:comp_min_pt_crit}
        \tilde{\mathfrak{m}} \ge \alpha_2 \delta^2 + \max (E_s (\theta_0), E_s (\theta_s)).
    \end{equation}
    
    \medskip
    
    \underline{Step 2} : Constructing a good minimizing sequence of paths.
    
    Take a minimizing sequence $\overline{c}_n$ for problem \ref{def:tilde_m}.
    Let $f_2 \in \mathcal{C}^1 (\mathbb{R})$ such that
    \begin{itemize}
        \item $f_2 (x) = x$ for all $x \in [- \frac{\pi}{2}, \frac{\pi}{2}]$,
        \item $f_2 (x) = \frac{\pi}{2}$ for all $x \geq \frac{9 \pi}{16}$,
        \item $f_2 (x) = - \frac{\pi}{2}$ for all $x \leq - \frac{9 \pi}{16}$,
        \item $f_2$ is $1$-Lipschitz and $\norm{f_2}_{L^\infty} \leq \frac{9 \pi}{16}$.
    \end{itemize}
    Then $\widetilde{c}_n (\lambda) \coloneqq f_2 (\overline{c}_n (\lambda)) \in \mathfrak{P}$ and for all $\lambda \in [0, 1]$, $E_s (\widetilde{c}_n (\lambda)) \le E_s (\overline{c}_n (\lambda))$, with a similar proof as in that of Proposition \ref{prop:strict monotony}. Therefore, $\widetilde{c}_n$ is also a minimizing sequence, with the further property that $\norm{\widetilde{c_n} (\lambda)}_{L^\infty} \leq \frac{9 \pi}{16}$.

    From this sequence, we can use the theory developed for the mountain pass theorem (see for instance \cite[Chapter~5,~Section~5]{aubin2006applied}). To be more precise, since $E_s$ is Gâteaux-differentiable and $E_s'$ is strong-to-weak$*$ continuous (\textit{i.e.} continuous from $W$ to $H^{-1}$), applying Corollary 3.2 of \cite{aubin2006applied} (see also Theorem 5.5) with $\widetilde{c_n}$ as reference sequence yields another sequence of paths $c_n \in \mathfrak{P}$ as well as some sequence of moments $\lambda_n$ such that 
    \begin{itemize}
        \item $\max_\lambda \norm{c_n (\lambda) - \widetilde{c}_n (\lambda)}_{H^1} \underset{n \rightarrow \infty}{\longrightarrow} 0$,
        \item $E_s (c_n (\lambda_n)) = \max_{\lambda \in [0, 1]} E_s (c_n (\lambda)) \underset{n \rightarrow \infty}{\longrightarrow} \tilde{\mathfrak{m}}$,
        \item $E_s' (c_n (\lambda_n)) \underset{n \rightarrow \infty}{\longrightarrow} 0$ in $H^{-1}$.
    \end{itemize}

    \medskip
    
    \underline{Step 3} : Convergence of the sequence $c_n (\lambda_n)$
    
    We know that $\tilde{\mathfrak{m}} < E_1 (\theta_0)$, so we can assume that $\sup_{n \in \N, \lambda \in [0, 1]} E_s ({c}_n (\lambda)) < E_1 (\theta_0) - \varepsilon$.
    Moreover, for all $\lambda \in [0, 1]$ and $n \in \N$,
    \begin{equation*}
        \norm{c_n (\lambda)}_{L^\infty} \leq \norm{\widetilde{c}_n (\lambda)}_{L^\infty} + \norm{c_n (\lambda) - \widetilde{c}_n (\lambda)}_{L^\infty} \leq \frac{9 \pi}{16} + C \norm{c_n - \widetilde{c}_n}_{\mathcal{C} ([0, 1], H^1)} \underset{n \rightarrow \infty}{\longrightarrow} \frac{9 \pi}{16},
    \end{equation*}
    so we can assume that $\norm{c_n}_{L^\infty} \leq \frac{5 \pi}{8}$.

    Since $E_s' (\theta) = (s \theta')' + \frac{s}{2} \sin (2 \theta)$, we proved that the sequence $c_n (\lambda_n)$ satisfies exactly the assumptions of Lemma \ref{lem:palais-smale}. Therefore, up to a sub-sequence, $c_n (\lambda_n)$ converges strongly to some critical point $\overline{\theta}$ in $H^1$. In particular, $E_s (c_n (\lambda_n))$ converges to $E_s (\overline{\theta})$, therefore $E_s (\overline{\theta}) = \tilde{\mathfrak{m}} < E_1 (\theta_*)$. Thus, with \eqref{eq:comp_min_pt_crit}, we know that $\overline{\theta}$ is neither $\theta_s$ or $\theta_0$.

    \medskip
    
    \underline{Step 4} : Contradiction with $\overline{\theta}$ being a strict local minimizer.
    
    Since $\overline{\theta}$ is a critical point of $E_s$, Lemma \ref{lem:crit_pt_strict_loc_min} also applies. Let $\overline{\varepsilon}, \overline{\alpha}_2 > 0$ given by this result and take $\overline{\delta} > 0$ such that
    \begin{equation*}% \label{eq:delta_cond_1}
        \overline{\delta} \le \frac{1}{2} \min \Bigl\{ \norm{\overline{\theta} - \theta_0}_{L^2}, \norm{\overline{\theta} - \theta_s}_{L^2} \Bigr\}
    \end{equation*}
    and
    \begin{equation*}% \label{eq:delta_cond_2}
        C \sqrt{\frac{2 E_1 (\theta_*)}{s_0}} \delta^\frac{1}{2} \le \overline{\varepsilon}.
    \end{equation*}
    Then, for any $n$ (possibly large enough),
    similar arguments as in step 1 shows that there exists $\overline{\lambda}_n$ such that $\norm{c_n (\overline{\lambda}_n) - \overline{\theta}}_{L^2} = \overline{\delta}$, and also that
    \begin{equation*}
        E_s (\overline{\theta}) + \overline{\alpha_2} \overline{\delta}^2
        \leq E_s (c_n (\overline{\lambda}_n))
        \leq E_s (c_n (\lambda_n))
        \underset{n \rightarrow \infty}{\longrightarrow} E_s (\overline{\theta}),
    \end{equation*}
    thus a contradiction.
\end{proof}

%%%%%%%%%%%%%%%%%%%%%%%%%%%%%%%%%%%%%%%%%%%%%%%%%%%%%%%%%%%%%%
%%%%%%%%%%%%%%%%%%%%%%%%%%%%%%%%%%%%%%%%%%%%%%%%%%%%%%%%%%%%%%
\subsection{Property of the solution}
Now that we have obtained existence and uniqueness of the solution, we prove here a property of the solution concerning its decay :

\begin{lemma}
Let $\theta_s$ be \emph{the} critical point of $E_s$ in the set $W$. Then
\begin{equation*}
    \forall\;x\in\R,\qquad\bigg|\big|\theta_s(x)\big|-\frac{\pi}{2}\bigg|\;\leq\; \pi \,\exp\bigg(-\int_0^{|x|}\frac{dy}{s(y)}\bigg).
\end{equation*}
\end{lemma}

\begin{proof}
By symmetric arguments, it is enough to consider the case $x\leq0$. We recall that
\begin{equation*}
    \big(s(x)\,\theta_s'(x)\big)'\;=\;-s(x)\frac{\sin(2\theta_s(x))}{2}.
\end{equation*}
Multiply this equation by $s(x)\,\theta_s'(x)$ and integrate gives
\begin{equation*}
    \frac{1}{2}\big|s(x)\,\theta_s'(x)\big|^2=-\int_{-\infty}^x\frac{\sin(2\theta_s(y))}{2}\theta_s'(y)\,s(y)^2 \diff y.
\end{equation*}
Since $-\pi/2\leq\theta_s(x)\leq0$ (with $x\leq0$), we use the convexity properties of the \emph{sine} function to write
\begin{equation*}
    -\frac{\sin(2\theta_s(y))}{2}\;\leq\;\frac{\pi}{2}+\theta_s(y).
\end{equation*}
Thus,
\begin{equation*}
    \frac{s(x)^2}{2}\big|\theta_s'(x)\big|^2\leq\int_{-\infty}^x\bigg(\frac{\pi}{2}+\theta_s(y)\bigg)\,\theta_s'(y)\,s(y)^2 \diff y.
\end{equation*}
Since the function $\theta_s$ is increasing, we can continue the estimate of this integral by writing $s(y)^2\leq1$. We eventually get
\begin{equation*}
    s(x)^2\Bigg[\bigg(\frac{\pi}{2}+\theta_s(x)\bigg)'\Bigg]^2\leq\frac{1}{2}\int_{-\infty}^x\Bigg[\bigg(\frac{\pi}{2}+\theta_s(y)\bigg)^{\!\!2}\Bigg]' \diff y=\bigg(\frac{\pi}{2}+\theta_s(x)\bigg)^2.
\end{equation*}
The signs of the quantities appearing above being known, we can remove the square and get:
\begin{equation*}
   \forall\;x\leq0,\qquad 0\leq\bigg(\frac{\pi}{2}+\theta_s(x)\bigg)'\leq\frac{1}{s(x)}\bigg(\frac{\pi}{2}+\theta_s(x)\bigg).
\end{equation*}
We know that $\theta_s(0) \leq \frac{\pi}{2}$ and then the announced estimate follows from the Grönwall lemma applied to the function $\frac{\pi}{2}+\theta_s.$
\end{proof}

\section*{Acknowledgements}

The author acknowledges partial support by the ANR project MOSICOF ANR-21-CE40-0004.

\appendix

\section{Continuity properties for paths in infinite dimension}\label{sec:append}

\begin{lemma}\label{lem:lipschitz lemma}
We consider the space $X_p:=\{f\in \dot{W}^{1,p}(\R)\;:\;f(0)=0\}$ for $p\in[1,+\infty].$ We define the multiplication by the signum function:
\begin{equation*}
    \begin{array}{cccc}
        S\;: & X_p & \longrightarrow & X_p \\
        \; & f & \longmapsto & \Big[x\mapsto\operatorname{sgn}(x)\,f(x)\Big].
    \end{array}
\end{equation*}
The functionnal $S$ is a well-defined linear involution on $X_p$ that satisfy $\|Sf\|_{L^p}=\|f\|_{L^p}$ (whether finite or infinite) and $\|\nabla(Sf)\|_{L^p}=\|\nabla f\|_{L^p}.$ In particular, it is a continuous endomophism on the Banach space $L^p\cap X_p$ (endowed with the standard $W^{1,p}$ norm) and on the Banach space $X_p$ endowed with the norm $f\mapsto\|\nabla f\|_{L^p}.$
\end{lemma}

\begin{proof}
First, we can check directly that $S$ is indeed a linear involution that preserves all the $L^p$ norms. We prove here that it maps $X_p$ into itself and preserves $\|\nabla f\|_{L^p}$.

Let $f:\R\to\R$ be a ${\mathcal C}^1$ function in $\dot{W}^{1,p}$ such that $f(0)=0$. It is direct to check that $x\mapsto\operatorname{sgn}(x)\,f(x)$ is ${\mathcal C}^1$ on $\R\setminus\{0\}$ and
\begin{equation*}
    \forall\,x\neq 0,\quad|\nabla f(x)|=|\nabla \operatorname{sgn}(x)\,f(x)|.
\end{equation*}
Concerning what happens at $x=0,$
since we have $f(0)=0$, then $Sf(0)=0$ and it is continuous at $x=0$ as a consequence of $f$ being continuous. 
Since $f$ is absolutely continuous, so is $Sf$ and its weak derivative is given by the following formula:
\begin{equation}\label{machine}
    \forall\,x\in\R,\quad\nabla(Sf)(x)= \operatorname{sgn}(x)\,\nabla f(x).
\end{equation}
It is indeed direct to check that
\begin{equation*}
    \forall\,x\in\R,\quad Sf(x)=\int_0^x\operatorname{sgn}(y)\,\nabla f(y) \diff y
\end{equation*}
We can then conclude from~\eqref{machine} that $\|\nabla(Sf)\|_{L^p}=\|\nabla f\|_{L^p}.$
\end{proof}

\begin{lemma}[Ponctual estimates for composition]\label{lem:composition ponctual estimates}
Let $\Omega$ be an interval of $\R$ and let $\Gamma_0,\Gamma_1$ be two intervals of $\R$ and let $A$ be an increasing ${\mathcal C}^0$ homeomorphism that maps $\Gamma_0$ into $\Gamma_1$. 
Define for $f,g$ two measurable increasing functions that map $\Omega\to\Gamma_0$:
\begin{equation*}
    \forall\,\lambda\in[0,1],\quad\forall\,x\in\Omega,\qquad T_\lambda(x):=A^{-1}\Big[\lambda A\circ f(x)+(1-\lambda) A\circ g(x)\Big],
\end{equation*}
where $\lambda\in[0,1]$ is fixed. Then we have, for all $p\in[1,+\infty)$:
\begin{equation}\label{eq:ponctual estimate L p}
    \forall\,x\in\Omega,\qquad\big|T_\lambda(x)-\psi(x)\big|^p\;\leq\;\big|f(x)-\psi(x)\big|^p+\big|g(x)-\psi(x)\big|^p,
\end{equation}
where $\psi$ is any measurable function.

If moreover $f$ and $g$ are non-decreasing and $A$ and $A^{-1}$ are $\mathcal{C}^1$ in the interior of\,
$\Gamma_0$ or $\Gamma_1$ respectively, then:
\begin{equation}\label{eq:ponctual estimate W 1 p}
        \forall\,x,y\in\Omega,\qquad\big|T_\lambda(x)-T_\lambda(y)\big|^p\;\leq\;2^{p-1}\Big(\big|f(x)-f(y)\big|^p+\big|g(x)-g(y)\big|^p\Big)
\end{equation}
\end{lemma}
\begin{proof}
\textit{Part 1.} Let $x\in\Omega$. Using the layer-cake representation of functions~\cite{Lieb_Loss_2014},
\begin{equation*}\begin{split}
    |T_\lambda(x)-\psi(x)|^p&=p\int_0^{+\infty}\nu^{p-1}\mathbbm{1}\big(\{|T_\lambda(x)-\psi(x)|>\nu\}\big) \diff \nu\\
    &=p\int_0^{+\infty}\nu^{p-1}\mathbbm{1}\Big(\big\{|A^{-1}\big(\lambda A\circ f+(1-\lambda) A\circ g\big)(x)-\psi(x)|>\nu\big\}\Big) \diff \nu,
\end{split}
\end{equation*}
where $\mathbbm{1}$ refers to the indicator functions. We treat separately the two cases depending on the sign of $T_\lambda(x)-\psi(x)$. We are led to :
\begin{equation}\label{double layer cake}\begin{split}
        &\qquad|T_\lambda(x)-\psi(x)|^p\\&=p\int_0^{+\infty}\nu^{p-1}\mathbbm{1}\Big(\big\{A^{-1}\big(\lambda A\circ f+(1-\lambda) A\circ g\big)(x)-\psi(x)>\nu\big\}\Big) \diff \nu\\
        &\qquad+p\int_0^{+\infty}\nu^{p-1}\mathbbm{1}\Big(\big\{A^{-1}\big(\lambda A\circ f+(1-\lambda) A\circ g\big)(x)-\psi(x)<-\nu\big\}\Big) \diff \nu\\
    &=p\int_0^{+\infty}\nu^{p-1}\mathbbm{1}\Big(\big\{\lambda A\circ f(x)+(1-\lambda) A\circ g(x)>A(\psi(x)+\nu)\big\}\Big) \diff \nu\\
        &\qquad+p\int_0^{+\infty}\nu^{p-1}\mathbbm{1}\Big(\big\{\lambda A\circ f(x)+(1-\lambda) A\circ g(x)<A(\psi(x)-\nu)\big\}\Big) \diff \nu.  
    \end{split}
\end{equation}
We now use the general fact, since $\lambda\in[0,1]$,
\begin{equation*}\lambda|A\circ f(x)|+(1-\lambda)|A\circ g(x)|>\mu\;\;\Longrightarrow\;\;|A\circ f(x)|>\mu\;\text{ or }\;|A\circ g(x)|>\mu,
\end{equation*}
for any $\mu\in\R$. We also know that for $X,Y$ two subsets of $\R,$ we have $\mathbbm{1}(X\cup Y)\leq\mathbbm{1}(X)+\mathbbm{1}(Y).$ 
Then
\begin{equation*}\begin{split}
    \mathbbm{1}\Big(\big\{\lambda &A\circ f(x)+(1-\lambda) A\circ g(x)>A(\psi(x)+\nu)\big\}\Big)\\
    &\leq\mathbbm{1}\Big(\big\{A\circ f(x)>A(\psi(x)+\nu)\big\}\Big)+\mathbbm{1}\Big(\big\{A\circ g(x)>A(\psi(x)+\nu)\big\}\Big)\\
    &=\mathbbm{1}\Big(\big\{f(x)>\psi(x)+\nu\big\}\Big)+\mathbbm{1}\Big(\big\{ g(x)>\psi(x)+\nu\big\}\Big).
    \end{split}
\end{equation*}
Where the last equality is given by the inversibility of $A$. Similarly,
\begin{equation*}\begin{split}
    \mathbbm{1}\Big(\big\{\lambda &A\circ f(x)+(1-\lambda) A\circ g(x)<A(\psi(x)-\nu)\big\}\Big)\\
    &\leq\mathbbm{1}\Big(\big\{f(x)<\psi(x)-\nu\big\}\Big)+\mathbbm{1}\Big(\big\{ g(x)<\psi(x)-\nu\big\}\Big).
    \end{split}
\end{equation*}
Plugging these estimates in~\eqref{double layer cake} eventually gives
\begin{equation*}\begin{split}
    |T_\lambda(x)-\psi(x)|^p    &\leq p\int_0^{+\infty}\nu^{p-1}\mathbbm{1}\Big(\big\{f(x)>\psi(x)+\nu\big\}\Big) \diff \nu+p\int_0^{+\infty}\nu^{p-1}\mathbbm{1}\Big(\big\{ g(x)>\psi(x)+\nu\big\}\Big) \diff \nu\\
        &\qquad+p\int_0^{+\infty}\nu^{p-1}\mathbbm{1}\Big(\big\{f(x)<\psi(x)-\nu\big\}\Big) \diff \nu+p\int_0^{+\infty}\nu^{p-1}\mathbbm{1}\Big(\big\{ g(x)<\psi(x)-\nu\big\}\Big) \diff \nu\\
        &=+p\int_0^{+\infty}\nu^{p-1}\mathbbm{1}\Big(\big\{ |f(x)-\psi(x)|>\nu\big\}\Big) \diff \nu+p\int_0^{+\infty}\nu^{p-1}\mathbbm{1}\Big(\big\{ |g(x)-\psi(x)|>\nu\big\}\Big) \diff \nu\\
        &=|f(x)-\psi(x)|^p+|g(x)-\psi(x)|^p,
    \end{split}
\end{equation*}

\medskip

\textit{Part 2.} Let $x,y\in\Omega$ with $x\leq y$. Assume that $f,g, A$ are increasing. We then have,
\begin{equation*}
    \big|T_\lambda(y)-T_\lambda(x)\big|=\int_{\lambda A\circ f(x)+(1-\lambda)A\circ g(x)}^{\lambda A\circ f(y)+(1-\lambda)A\circ g(y)}\frac{d}{dt}A^{-1}(t) \diff t.
\end{equation*}
We can estimate this integral as follows, since $\lambda\in[0,1]$:
\begin{equation*}
\int_{\lambda A\circ f(x)+(1-\lambda)A\circ g(x)}^{\lambda A\circ f(y)+(1-\lambda)A\circ g(y)}\frac{d}{dt}A^{-1}(t) \diff t\leq\int_{\min\{A\circ f(x);\, A\circ g(x)\}}^{\max\{A\circ f(y) ;\,A\circ g(y)\}}\frac{d}{dt}A^{-1}(t) \diff t.
\end{equation*}
We continue this estimate by adding a non-negative term and reorganize the bounds of the integrals: 
\begin{equation*}%\label{cut 1}
    \begin{split}
    &\qquad\qquad\int_{\min\{A\circ f(x);\, A\circ g(x)\}}^{\max\{A\circ f(y) ;\,A\circ g(y)\}}\frac{d}{dt}A^{-1}(t) \diff t\\
    &\leq\int_{\min\{A\circ f(x);\, A\circ g(x)\}}^{\max\{A\circ f(y) ;\,A\circ g(y)\}}\frac{d}{dt}A^{-1}(t) \diff t+\int_{\max\{A\circ f(x);\, A\circ g(x)\}}^{\min\{A\circ f(y) ;\,A\circ g(y)\}}\frac{d}{dt}A^{-1}(t)\,\\
    &=\int_{A\circ f(x)}^{A\circ f(y)}\frac{d}{dt}A^{-1}(t) \diff t+\int_{A\circ g(x)}^{A\circ g(y)}\frac{d}{dt}A^{-1}(t) \diff t\\
    &=\big|f(y)-f(x)\big|+\big|g(y)-g(x)\big|.
\end{split}
\end{equation*}
As a consequence, using the standard convexity inequality for power functions:
\begin{equation*}
    \big|T_\lambda(y)-T_\lambda(x)\big|^p\;\leq\;2^{p-1}\Big(\big|f(x)-f(y)\big|^p+\big|g(x)-g(y)\big|^p\Big) \qedhere
\end{equation*}
\end{proof}

\begin{corollary}\label{coro:stability for L p} With the same assumptions:

\medskip

$(i)$ If $f$ and $g$ belong to the affine space $\psi+L^p(\Omega;\Gamma_0)$ with $p<+\infty$ (endowed with the standard $L^p$ norm) then the function $\lambda\mapsto T_\lambda$ draws a continuous path with respect to the $L^p$ topology that connects $f$ to $g$.

\medskip

$(ii)$ If $f$ and $g$ belong to $\dot{W}^{1,p}(\Omega;\Gamma_0)$ with $p<+\infty$ (endowed with the standard $\dot{W}^{1,p}$ half-norm) and are increasing, then the function $\lambda\mapsto T_\lambda$ draws a continuous path with respect to the $\dot{W}^{1,p}(\Omega;\Gamma_0)$ topology that connects $f$ to $g$. 
\end{corollary}
\begin{proof}
First, by continuity of $A^{-1}$, we have
\begin{equation*}
    \forall\,x\in\Omega,\qquad T_\lambda(x)\longrightarrow T_\mu(x),\quad\text{as }\lambda\to\mu.
\end{equation*}
As a consequence of~\eqref{eq:ponctual estimate L p}, we can apply the Lebesgue dominated convergence theorem and conclude:
\begin{equation*}
    \int_\Omega\big |T_\lambda(x)-T_\mu(x)\big|^p\diff x\longrightarrow0,\quad\text{as }\lambda\to\mu.
\end{equation*}
Similarly, by continuity of $(A^{-1})'$ on $\Gamma_1$ and of $A'$ on $\Gamma_0$, we also have
\begin{equation*}
    \forall\,x\in\Omega,\qquad \nabla T_\lambda(x)\longrightarrow \nabla T_\mu(x),\quad\text{as }\lambda\to\mu.
\end{equation*}
Since \eqref{eq:ponctual estimate W 1 p} implies that $|\nabla T_\lambda(x)|^p\lesssim|\nabla f(x)|^p+|\nabla g(x)|^p$ for any $\lambda$ and $x$, the Lebesgue dominated convergence theorem also applies to the gradients:
\begin{equation*}
    \int_\Omega\big |\nabla T_\lambda(x)-\nabla T_\mu(x)\big|^p\diff x\longrightarrow0,\quad\text{as }\lambda\to\mu. \qedhere
\end{equation*}
\end{proof}
This corollary gives the continuity of the path $\lambda\mapsto T_\lambda$ with respect to the $\dot{W}^{1,p}$ topology. In the case of the $L^p$ topology, it is possible to improve the result and get that this path is actually ${\mathcal C}^1$ using the implicit functions theorem:

\begin{lemma}
Let $\Omega$ be an interval of $\R$ and let $\Gamma_0,\Gamma_1$ be two open sets of $\R^p$ such that there exists a ${\mathcal C}^1$ diffeomorphism $A$ that maps $\Gamma_0$ into $\Gamma_1$. Assume that $\Gamma_1$ is convex. Let $X\subseteq L^p(\Omega;\Gamma_0)$ be a Banach space such that for all function $\varphi\in X$ we have  $A\circ \phi$ in the affine space $\psi_0+L^2(\Omega;\Gamma_1),$ for some function $\psi_0$.  We further assume that the map $\phi\in X\mapsto A\circ \phi-\psi_0\in L^2(\Omega;\Gamma_1)$ is ${\mathcal C}^1$. We now define
\begin{equation*}
        \forall\,\lambda\in\R,\quad\forall\,x\in\Omega,\qquad T_\lambda(x):=A^{-1}\Big[\lambda A\circ f(x)+(1-\lambda) A\circ g(x)\Big],
\end{equation*}
where $f,g$ are two fixed functions in $X$. Then the function $\lambda\mapsto T_\lambda$ with values in $X$ is a ${\mathcal C}^1$ function (where $X$ is endowed with a $L^p$ or $\dot{W}^{1,p}$ norm).
\end{lemma}
\begin{proof}
Let $f$ and $g$ be two functions in $X$. We first remark that for all $\lambda\in\R$ we have $T_\lambda\in X$, as a consequence of $\Gamma_1$ convex and Lemma~\ref{lem:composition ponctual estimates}. We recall that every separable Hilbert space admits a Hilbertian base (see Theorem V.10 in~\cite{Brezis_2010}). 
In particular, any subspace of $L^2$ admits a Hilbertian base (associated to the standard scalar product of $L^2$). 
We denote by $(e_n)_{n\in\N}$ a Hilbertian base of the space $\{(A\circ f-A\circ g)\}^\perp\subseteq L^2(\Omega;\Gamma_1).$
\begin{equation*}
    \begin{array}{cccc}
        {\mathcal F}: &\; X\; & \longrightarrow & \;\ell^2(\R) \\
         \,& \phi & \longmapsto &  \left<A\circ\phi-A\circ f \,|\, e_n\right>_{L^2}.
         \end{array}
\end{equation*}
It is direct to check that ${\mathcal F}(\phi)=0$ if and only if $A\circ\phi-A\circ f$ belongs to the affine straight line $A\circ g+(A\circ f-A\circ g)\R$. Namely,
\begin{equation*}
    {\mathcal F}(\phi)=0\qquad\Longleftrightarrow\qquad\exists\;\lambda\in\R,\quad A\circ\phi=\lambda A\circ f+(1-\lambda)A\circ g.
\end{equation*}
This last equality is equivalent to $\phi=T_\lambda$. To conclude to the regularity of $\lambda\mapsto T_\lambda$ we use the implicit functions theorem applied to the function ${\mathcal F}$. 
The fact that $A$ is a ${\mathcal C}^1$ diffeomorphism ensures that the differential of ${\mathcal F}$ is not degenerate.
\end{proof}

\bibliographystyle{abbrv} 
{\bibliography{biblio}}

\end{document}